\documentclass[a4paper]{article}
\usepackage{amsthm,amssymb,amsmath,enumerate}
\usepackage[shortlabels]{enumitem}
\usepackage{hyperref} 
\hypersetup{
	colorlinks=true,  
	linkcolor=black,    
citecolor=black,
	urlcolor=black      
}

\usepackage{tikz}
\usetikzlibrary{calc}
\usetikzlibrary{backgrounds}

\colorlet{auchblau}{blue!60!white}
\colorlet{hellblau}{blue!40!white}
\colorlet{hellrot}{red!70!white}
\colorlet{hellgrau}{black!20!white}
\colorlet{nochhellergrau}{black!20!white}
\colorlet{dunkelgrau}{black!60!white}

\tikzstyle{hvertex}=[thick,circle,inner sep=0.cm, minimum size=2mm, fill=white, draw=black]
\tikzstyle{hedge}=[very thick]
\tikzstyle{medge}=[thick]
\tikzstyle{harrow}=[thick,arrows=->]
\tikzstyle{darrow}=[thick,arrows=<-]
\tikzstyle{rededge}=[very thick,red]
\tikzstyle{point}=[draw,circle,inner sep=0.cm, minimum size=1mm, fill=black]
\tikzstyle{pointer}=[thick,->,shorten >=2pt,color=dunkelgrau]
\tikzstyle{facebdry}=[color=auchblau, very thick] 
\tikzstyle{face}=[facebdry,fill=hellblau]
\tikzstyle{nface}=[color=hellblau,fill=hellblau,thick] 
\tikzset{>={latex}}
\tikzstyle{tinyvx}=[thick,circle,inner sep=0.cm, minimum size=1.3mm, fill=white, draw=black]
\tikzstyle{smallvx}=[hvertex,minimum size=1.7mm]

\newcommand{\blase}[5]{\begin{scope}
		\draw[fill=nochhellergrau,rotate=#2, shift = {#1}, line width=1pt] (0,0) .. controls ++(0.4*#3,-0.7*#4) and ++(0,-0.8*#4) .. (#3,0); 
		\draw[fill=nochhellergrau,rotate=#2, shift = {#1}, line width=1pt] (0,0) .. controls ++(0.4*#3,0.7*#4) and ++(0,0.8*#4) .. (#3,0);
		\node at ($(#2:#3/2)+#1$){#5};
\end{scope}}
\newcommand{\stengelblase}[6]{\coordinate (verschiebung) at ($#1+(#3:#2)$){}; \draw[hedge] #1 -- (verschiebung); \blase{(verschiebung)}{#3}{#4}{#5}{#6} }
\newcommand{\buendel}[2]{\draw[medge,bend left=30] #1 to #2; \draw[medge,bend left=10] #1 to #2; \draw[medge,bend right=30] #1 to #2; \draw[medge,bend right=10] #1 to #2;}

\newtheorem{definition}{Definition}
\newtheorem{proposition}[definition]{Proposition}
\newtheorem{theorem}[definition]{Theorem}

\newtheorem{lemma}[definition]{Lemma}
\newtheorem{observation}[definition]{Observation}
\newtheorem{conjecture}[definition]{Conjecture}
\newtheorem{question}[definition]{Question}

\newcommand{\N}{\mathbb N}

\newcommand{\es}{\emptyset}

\renewcommand{\phi}{\varphi}
\renewcommand{\epsilon}{\varepsilon}

\newcommand{\COMMENT}[1]{}
\renewcommand{\COMMENT}{\footnote} 

\newcommand{\comment}[1]{}

\newcommand{\emtext}[1]{\text{\em #1}}

\newcommand{\mymargin}[1]{%
  \marginpar{%
    \begin{minipage}{\marginparwidth}\small%
      \begin{flushleft}%
        #1%
      \end{flushleft}%
   \end{minipage}%
  }%
}%

\newcommand{\sm}{\setminus}

\newcommand{\pw}{{\rm pw}}
\newcommand{\EP}{Erd\H os-P\'osa property}
\newcommand{\cF}{\mathcal{F}}

\newcommand{\cP}{\mathcal{P}}
\newcommand{\wmin}{\omega_{min}}

\newcommand{\lev}{L}
\newcommand{\leT}{\le_T}

\newcommand{\Smin}{L_{15}^{\min}}

\title{The edge-\EP}
\author{Henning Bruhn\thanks{Partially supported by DFG, grant no.\  BR 5449/1-1.} \and Matthias Heinlein \and Felix Joos\thanks{Partially supported by DFG, grant no.~JO~1457/1-1.}}
\date{}

\begin{document}

\maketitle

\begin{abstract}
Robertson and Seymour proved that 
the family of all graphs containing a fixed graph $H$ as a minor 
has the \EP\ if and only if $H$ is planar.
We show that this is no longer true for the edge version of the \EP,
and indeed even fails when $H$ is an arbitrary subcubic tree of large pathwidth
or a long ladder. 
This answers a question of Raymond, Sau and Thilikos.
\end{abstract}

\section{Introduction}

Duality is arguably
 one of the most fundamental concepts in combinatorial optimisation and beyond.
In graph theory, we encounter it in Menger's theorem:
Given two vertex sets $A,B$ in a graph $G$, then the maximal number of vertex-disjoint $A$--$B$-paths in $G$
equals the minimal number of vertices in $G$ meeting all $A$--$B$-paths.

If we ask for different target objects instead of $A$--$B$-paths,
then this minimum maximum duality might break down: 
For example, in many graphs
the maximal number of vertex-disjoint edges (size of a maximum matching) does not
 coincide with the minimal number of vertices meeting all edges (size of a minimum vertex cover).
The duality principle, however, is not completely lost.
Indeed, the size of a minimum vertex cover is bounded 
from above and below by a function of the size of a maximum matching,
and vice versa.

Erd\H{o}s and P\'osa found a similar duality for cycles~\cite{EP65}.
Suppose a graph~$G$ contains at most $k$ vertex-disjoint cycles,
then there is a vertex set of size $O(k\log k)$ meeting all cycles.
More generally, we say a set of graphs $\cF$ has the \emph{\EP}{} if there exists a function $f\colon\N\to\mathbb R$ such that for every graph $G$ and every integer $k$,
there are $k$ vertex-disjoint graphs in $G$ each isomorphic to a member of $\cF$
or there is a vertex set $X$ of $G$ of size at most $f(k)$ meeting all subgraphs in $G$ isomorphic to a graph in $\cF$.
Thus, the set of cycles has the \EP{}.
By now many families of graphs are known to have the property.
In particular, this includes various selections of cycles (parity and length constraints~\cite{BBR07,BJS14,MNSW17,Tho88}, rooted cycles~\cite{PW12}, and group-labelled cycles~\cite{HJW16}).

One of the most striking results in the area is due to Robertson and Seymour.
Say that a graph is an \emph{$H$-expansion} if the graph $H$ is its minor.  

\begin{theorem}[Robertson and Seymour \cite{RS86}] \label{metaThm}
The family of $H$-expansions has the \EP\ if and only if $H$ is a planar graph.
\end{theorem}

Theorem~\ref{metaThm} is very general.
It extends  Erd\H{o}s and P\'osa's original result,
but it also provides a non-topological characterisation for planar graphs.
The high level proof technique of Robertson and Seymour, which
uses tools from their
Graph Minor project, has inspired many later authors.

So far, we have discussed only vertex-disjoint target graphs.
It is, however, equally natural to ask for edge-disjoint target graphs
or, alternatively, for an edge set meeting all target subgraphs.
We define the \emph{edge-\EP}{}  by replacing in the definition every occurrence of ``vertex'' by ``edge''.
Menger's theorem as well as Erd\H{o}s and P\'osa's theorem
have edge analogues in this sense.
While the literature  on the (vertex-)\EP{} is extensive,
we are only aware of a  small number of results on the edge-\EP{}.

It is easy to marginally modify the approach of 
Robertson and Seymour's proof of Theorem~\ref{metaThm} to show that $H$-expansions do 
not have the edge-\EP{} whenever $H$ is non-planar. 
Perhaps motivated by this observation, 
Raymond, Sau, and Thilikos~\cite{RST13} asked whether Theorem~\ref{metaThm} 
holds also in the ``edge'' version:

\begin{question}\label{edgeMeta}
Do $H$-expansions have the edge-\EP{} whenever $H$ is a planar graph?
\end{question}

There are partial answers.
For $H=K_3$, the question defaults to the edge version of the result of Erd\H{o}s and P\'osa.
If $H$ is a theta graph, a multigraph consisting of $r$ parallel edges,
then the answer is also ``yes''~\cite{RST13}.
Answering a question of Birmel\'e, Bondy, and Reed,
we showed that long cycles (cycles of length at least $\ell$ for some $\ell\in \N$) have the edge-property~\cite{BHJ17}.
Another family with the edge-property are $K_4$-expansions; see~\cite{BH18}.
While the first two results can be deduced from 
their corresponding vertex versions with not too much effort,
the proofs of the latter two results are involved and seem to require several new techniques.

The aim of this article is to show that,
nevertheless, there are significant
differences between the edge-\EP{} and the (vertex-)\EP{}:
\begin{theorem}\label{mainthm}
The family of $H$-expansions does not have the edge-Erd\H os-P\'osa property
if
\begin{enumerate}[\rm (i)]
\item $H$ is a ladder of length at least~$71$, or
\item $H$ is a subcubic tree of pathwidth at least~$19$.
\end{enumerate}
\end{theorem}

We remark that Theorem~\ref{mainthm} not only shows that there is some tree but that for
all subcubic trees $H$ that are not too path-like,
the family of $H$-expansions does not have the edge-\EP.
Hence Question~\ref{edgeMeta} has a negative answer even if we restrict our attention to trees.
For the sake of a cleaner presentation of the argument, we make no attempt to optimise the constants in Theorem~\ref{mainthm}.

To verify that a certain family does not have the \EP, normally 
an obstruction is constructed that certifies this: a graph (or rather a graph family)
that does not admit two disjoint target graphs but that necessitates
an arbitrarily large vertex set meeting all target subgraphs.
Interestingly, these obstructions all follow a common pattern. 
They usually consist of a large grid (or wall), with a certain gadgets attached 
to the boundary of the grid. In all cases
known to us, it is straightforward to check that the obstruction
works as intended.

Our key contribution is  an entirely new type of obstruction (Section~\ref{sec:fromK4toLinkages}).
This type does not contain a large grid or wall, 
and it is technically involved to verify that these graphs are indeed obstructions.

\section{Linkages}
	\label{sec:fromK4toLinkages}

The proof of our main theorem is based 
on the insight that linkages between four terminals 
do not have the edge-\EP, not even if the ambient graph 
has small treewidth. In particular, we will construct a
wall-like structure of small treewidth, in which linkages 
fail to have the edge-property.

For our paper it is not important how the treewidth or pathwidth $\pw(G)$ of a graph $G$ is defined
but only that treewidth and pathwidth measure how tree-like and path-like $G$ is.
A formal definition can be found in the introduction of Robertson and Seymour's article \cite{RS83} and in many textbooks.

As a warm-up, and because it will lead us to the edge version,
we show that linkages do not have the ordinary \EP\ either. 
For vertex sets $A,B,C,D$, an \emph{($A$--$B$, $C$--$D$)-linkage} is the disjoint union 
of an $A$--$B$-path with an $C$--$D$-path. 
Suppose that ($A$--$B$, $C$--$D$)-linkages  have the \EP, and suppose 
that every graph $G$ that does not contain two \mbox{(vertex-)}disjoint
($A$--$B$, $C$--$D$)-linkages admits a set of at most $r$ vertices meeting 
every ($A$--$B$, $C$--$D$)-linkage.

		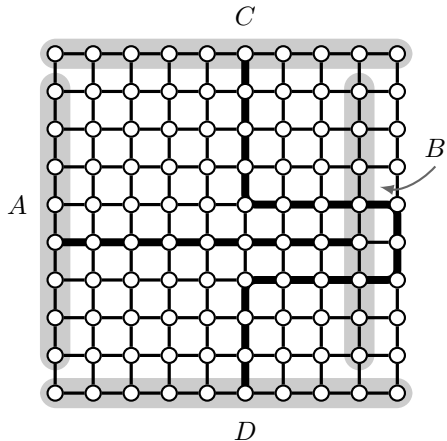
\begin{figure}[bht]
			\centering
			\begin{tikzpicture}[scale=0.5]
			\def\size{10};
			\pgfmathtruncatemacro{\sizeMtwo}{\size-2}
\tikzstyle{ABCD}=[rounded corners=6pt, fill = hellgrau,draw=none]
\tikzstyle{linkage}=[hedge,line width = 3pt,draw=black,rounded corners]			
			\draw[ABCD]
			(0.6,1.6) rectangle ++(0.8,\size-2-0.1);
			\node at (0,\size/2+1){$A$};
			
			\draw[ABCD]
			(0.6,0.6) rectangle ++(\size-0.2,0.8);
			\node at (\size/2+1,0){$D$};
			
			\draw[ABCD]
			(0.6,\size -1 + 0.6) rectangle ++(\size-0.2,0.8);
			\node at (\size/2+1,\size+1){$C$};
			
			\draw[ABCD]
			(\size-2+0.6,1.6) rectangle ++(0.8,\size-2-0.1);
			\node at (\size+1,\size/2+2.5) (B) {$B$};
			
			\draw[linkage] (6,10)--(6,6)--(10,6)--(10,4)--(6,4)--(6,1);
			\draw[linkage] (1,5)--(9,5);
			
			\foreach \r in {1,...,\size}{
				\foreach \c in {1,...,\size}{
					\node[hvertex] (x\r_\c) at (\c,\r){};
				}
				
			}
			\pgfmathsetmacro{\sizeM}{\size-1}
			\foreach \r in {1,...,\size}{
				\foreach \c in {1,...,\sizeM}{
					\pgfmathtruncatemacro{\cc}{\c+1}
					\draw[hedge] (x\r_\c) to (x\r_\cc);
				}	
			}
			\pgfmathsetmacro{\sizeM}{\size-1}
			\foreach \r in {1,...,\sizeM}{
				\pgfmathtruncatemacro{\rr}{\r+1}
				\foreach \c in {1,...,\size}{
					\draw[hedge] (x\r_\c) to (x\rr_\c);
				}	
			}

\draw[pointer, bend left=30] (B.south) to (\size-0.6,\size/2+1.5);
			\end{tikzpicture}
			\caption{An $4r \times 4r$ grid that neither contains two disjoint ($A$--$B$, $C$--$D$)-linkages nor a set of size of most $r$ meeting every ($A$--$B$, $C$--$D$)-linkage.}
			\label{fig:grid}
		\end{figure}
	
		Let $G$ be a $4r\times4r$-grid and 
		let the sets $A,B,C,D$ be chosen as in Figure~\ref{fig:grid}.
		It is easy to check that no set of at most $r$ vertices  intersects every linkage.
		Suppose that the graph contains two disjoint ($A$--$B$, $C$--$D$)-linkages.
		Then these two linkages contain two disjoint $C$--$D$-paths
		but  
		at most one of them can contain a vertex of the rightmost column
		of the grid.
		However, every $C$--$D$-path that does not contain a vertex of the rightmost column separates $A$ and $B$,
		and thus meets every $A$--$B$-path.
		Therefore, the graph does not contain two disjoint linkages.
We obtain:
\begin{proposition}
($A$--$B$, $C$--$D$)-linkages do not have the \EP.
\end{proposition}

Now we consider edge-disjoint linkages. 
By replacing the grid by a wall\footnote{For a formel definition of a wall, see for example~\cite{HJW16}.} in Figure~\ref{fig:grid} 
we immediately see that linkages do not have the edge-\EP\ either.
To prove Theorem~\ref{mainthm}, however, we need a stronger statement. 

As 
a graph of large enough treewidth contains an $H$-expansion of any 
planar graph $H$, we cannot use the  construction of Figure~\ref{fig:grid},
which has large treewidth. Rather, we modify the construction in such a way
that the resulting graph has even small pathwidth, but still shows 
that linkages do not have the edge-\EP.

We simplify a bit and consider linkages between single-vertex sets. 
That is,  we are interested in 
\emph{($a$--$b$, $c$--$d$)-linkages} for vertices $a,b,c,d$, the disjoint union
of an $a$--$b$-path with an $c$--$d$-path.

Let $r$ be a positive integer. 
A \emph{condensed wall} $W$ of size $r$ is defined as follows (see Figure~\ref{fig:condWall} for an illustration):
\begin{itemize}
	\item For every $j\in [r]$, let $P^j=u^j_1,\ldots, u^j_{2r}$ be a path of length $2r-1$ and 
	for $j\in \{0\}\cup [r]$, let $z^j$ be a vertex. Moreover, let $a,b$ two further vertices.
	\item For every $i,j\in [r]$, add the edges $z^{j-1}u^{j}_{2i}$, $z^ju^{j}_{2i-1}$, $z^{i-1}z^i$, $au^j_1$ and $bu^j_{2r}$.
\end{itemize}
We define $c=z^0$ and $d=z^r$ and refer to
 $$W_j=W[\{u_1^j,\ldots, u_{2r}^jz^{j-1},z^j\}]$$
as the \emph{$j$-th layer of $W$}.
The vertices $a,b,c,d$ are called the \emph{terminals} of $W$.

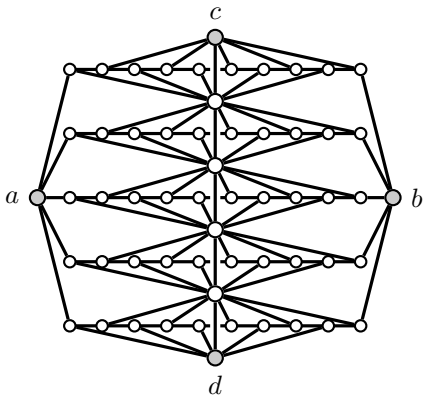
\begin{figure}[bht]
	\centering
		\begin{tikzpicture}[scale=0.85]
			\tikzstyle{tinyvx}=[thick,circle,inner sep=0.cm, minimum size=1.5mm, fill=white, draw=black]
			
			\def\vstep{1}
			\def\hstep{0.5}
			\def\hwidth{9}
			\def\hheight{4}
			
			\def\totalheight{\hheight*\vstep}
			\def\totalwidth{\hwidth*\hstep}
			\pgfmathtruncatemacro{\minustwo}{\hwidth-2}
			\pgfmathtruncatemacro{\minusone}{\hwidth-1}
			
			\foreach \j in {0,...,\hheight} {
				\draw[hedge] (0,\j*\vstep) -- (\hwidth*\hstep,\j*\vstep);
				\foreach \i in {0,...,\hwidth} {
					\node[tinyvx] (v\i\j) at (\i*\hstep,\j*\vstep){};
				}
			}
			
			\foreach \j in {1,...,\hheight}{
				\node[hvertex] (z\j) at (0.5*\hwidth*\hstep,\j*\vstep-0.5*\vstep) {};
			}
			\pgfmathtruncatemacro{\plusvone}{\hheight+1}
			
			\node[hvertex,fill=hellgrau,label=above:$c$] (z\plusvone) at (0.5*\totalwidth,\totalheight+0.5*\vstep) {};
			\node[hvertex,fill=hellgrau,label=below:$d$] (z0) at (0.5*\totalwidth,-0.5*\vstep) {};

			\foreach \j in {1,...,\plusvone}{
				\pgfmathtruncatemacro{\subone}{\j-1}
				\draw[line width=1.3pt,double distance=1.2pt,draw=white,double=black] (z\j) to (z\subone);
				\foreach \i in {1,3,...,\hwidth}{
					\draw[hedge] (z\j) to (v\i\subone);
				}
			}
			
			\foreach \j in {0,...,\hheight}{
				\foreach \i in {0,2,...,\hwidth}{
					\draw[hedge] (z\j) to (v\i\j);
				}
			}

			\pgfmathtruncatemacro{\minusvone}{\hheight-1}
			\node[hvertex,fill=hellgrau,label=left:$a$] (a) at (-\hstep,0.5*\totalheight) {};
			\foreach \j in {0,...,\hheight} {
				\draw[hedge] (a) -- (v0\j);
			}
			
			\node[hvertex,fill=hellgrau,label=right:$b$] (b) at (\totalwidth+\hstep,0.5*\totalheight) {};
			\foreach \j in {0,...,\hheight} {
				\draw[hedge] (v\hwidth\j) to (b);
			}
			\end{tikzpicture}
			\caption{A condensed wall of size 5.}
			\label{fig:condWall}
		\end{figure}

Condensed walls have their origin in the construction of Figure~\ref{fig:grid}.
If we replace the grid by a wall, 
contract each of $A,B,C,D$ to a single vertex, and then contract every 
second row of the wall we arrive at a graph that is basically equivalent 
to a condensed wall. 

We continue with a few observations about condensed walls.
Let $W$ be a condensed wall of size $r$.
Then $W-a-b$ is easily seen to have pathwidth (and hence treewidth) at most~$3$.
Therefore, $W$ has pathwidth at most $5$.

\begin{observation}\label{wallTreewidth}
A condensed wall (of any size) has pathwidth at most~$5$.
\end{observation}

Next, we  consider an ($a$--$b$, $c$--$d$)-linkage in $W$.
Observe that in any ($a$--$b$, $c$--$d$)-linkage the $a$--$b$-path is of the form $aP^jb$
as $\{a,b,z^j\}$ separates $c$ and $d$ and thus $\{z^0,\ldots,z^{r}\}$ is disjoint from the $a$--$b$-path\footnote{For a tree $T$ and $a,b\in V(T)$, there is a unique $a$--$b$-path in $T$ and this path is denoted by~$aTb$.}.
Moreover,
for any $j$,
every $c$--$d$-path in $W-E(aP^jb)$ that avoids $a,b$ contains the edge $z^{j-1}z^{j}$,
which shows that there do not exist two edge-disjoint linkages.

\begin{observation}\label{trees:wallNoTwo}
A condensed wall $W$ does not contain two edge-disjoint $(a$--$b,c$--$d)$-linkages.
\end{observation}
		
Suppose $X$ is a set of at most $r-1$ edges in $W$.
Then there exists some $j\in [r]$ such that $aP^jb$ and the $j$-th layer $W_j$ are edge-disjoint from $X$.
Moreover, for all $i\in [r]$, the vertices $z^{i-1}$ and $z^{i}$ belong to the same component in $W-X-aP^jb$.
Therefore, $W-X$ contains an $(a$--$b,c$--$d)$-linkage.

\begin{observation}\label{wallNoHit}
Suppose $W$ is a condensed wall of size $r$ and let $X$ be a set of at most $r-1$ edges.
Then, $W-X$ contains an $(a$--$b,c$--$d)$-linkage.
\end{observation}

\section{Long ladders}\label{sec:ladder}

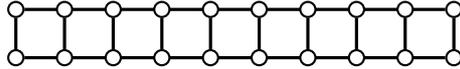
\begin{figure}[bht]
\centering
\begin{tikzpicture}[scale=0.8]

\def\runglen{0.8cm}
\def\ladderlen{10}

\foreach \i in {1,...,\ladderlen}{
  \node[hvertex] (a\i) at (\i*\runglen,0){};
  \node[hvertex] (b\i) at (\i*\runglen,\runglen){};
  \draw[hedge] (a\i) -- (b\i);
}
\begin{scope}[on background layer]
\draw[hedge] (a1) -- (a\ladderlen);
\draw[hedge] (b1) -- (b\ladderlen);
\end{scope}

\end{tikzpicture}
\caption{A ladder of length~$10$}\label{ladderfig}
\end{figure}
	A \emph{ladder} 
 is a graph $L$ with vertex set $V(L)=\{u_i,v_i: i\in [\ell]\}$ and
	edge set  $E(L)=\{u_iu_{i+1},v_iv_{i+1}: i\in [\ell-1]\}\cup \{u_iv_i:i\in [\ell]\}$ for some $\ell\in \N$.
	The edges $u_iv_i$ are called the \emph{rungs} of the ladder.
	We define the \emph{length} of a ladder as the number of its rungs.
	We prove the following, which implies Theorem~\ref{mainthm}~(i):

	\begin{theorem}\label{ladders:thm}
	Let $H$ be a ladder of length at least 71. 
	Then the subdivisions of~$H$ do not have the edge-Erd\H{o}s-P\'osa property.
	\end{theorem}
	
	Observe that whenever $H$ has maximum degree at most $3$, then a graph contains $H$ as a minor if and only if it contains a subdivision of $H$.
	Before we prove Theorem \ref{ladders:thm}, we note a small observation about ladders that are contained in a condensed wall.
	\begin{lemma}\label{ladders:seventeen}
		Let $W$ be a condensed wall of any size and with terminals $a,b,c,d$.
Then, $W-\{a,b\}$ does not contain any  subdivision of a ladder of length 6.
	\end{lemma}
	\begin{proof}
		Assume for a contradiction that $L$ is a subdivision of a ladder of length at least 6 in $W-\{a,b\}$.
		As $L$ is $2$-connected, $L$ is contained in a block of $W-\{a,b\}$, that is, in one layer $W_i$ of $H$. 
		As every cycle in $W_i$ contains $z^{i-1}$ or $z^{i}$,
the layer  $W_i$ does not contain three disjoint cycles.
A ladder of length~$6$, however, contains three disjoint cycles. 
This is a contradiction.
	\end{proof}
	
	\medskip
	
	\begin{proof}[Proof of Theorem \ref{ladders:thm}]
		We present the proof for the case 
		when the length of the ladder is exactly $71$
		as the proof is almost the same when the length is larger.
		Before we start let us give a short outline of the proof.
For a ladder $L$ of length~$71$, we remove
a rung at one third of the length and one at two thirds of the length
		and split the rest of $L$ into three parts of equal size.
		We glue inflations of these three parts to a large condensed wall to form a graph $G$.
Then, we prove that in every subdivision of $L$ in $G$
the removed rungs form an $(a$--$b,c$--$d)$-linkage in the condensed wall, 
which by construction and Observation~\ref{trees:wallNoTwo} proves the theorem.

\medskip

We start now with the proof.
Let $L$ be a ladder of length 71,
that is, we may write $V(L)=\{u_i,v_i: i\in [71]\}$ and $E(L)=\{u_iu_{i+1},v_iv_{i+1}: i\in [70]\}\cup \{u_iv_i:i\in [71]\}$.
Let $U$ denote all vertices $w$ of $L$ with $d_L(w)=3$,
that is, $U=\{u_i,v_i: 2\le i\le 70\}$.

Let $r\geq 2$ be an arbitrary positive integer. 
We construct a graph $G$ and a map $\epsilon:V(L)\to V(G)$ as follows:
		\begin{itemize}
			\item Start with a condensed wall $W$ of size $r$ with terminals $a,b,c,d$; 
			\item for every vertex $w\in U\setminus \{u_{24},v_{24},u_{48},v_{48}\}$, 
				add a new vertex $x$ to $G$ and set $\epsilon(w)=x$;
			\item set $\epsilon(u_{24})=a$, $\epsilon(v_{24})=b$, $\epsilon(v_{48})=c$, $\epsilon(u_{48})=d$;
			\item for every $U$-path $P$ between two vertices $s,t\in U$ such that $s\notin \{u_{24},v_{24},u_{48},v_{48}\}$,
				create $r$ internally disjoint $\epsilon(s)$--$\epsilon(t)$-paths of 
				length~$3$.
 		\end{itemize}
		We set $T=\{a,b,c,d\}$,
and we observe that $G-T$ has four components, $W-T$ and three others.
We denote by $A$ the component of $G-T$ that contains $\epsilon(u_2)$, 
by $B$ the one that contains $\epsilon(u_{25})$, 
and by $C$ the one that contains $\epsilon(u_{70})$.

\begin{figure}[bht]
\centering
\begin{tikzpicture}[scale=0.8]
\tikzstyle{ded}=[line width=0.8pt,double distance=1.2pt,draw=white,double=black]
\tikzstyle{bubble}=[color=hellgrau,line width=6pt,fill=hellgrau,rounded corners=4pt]
\def\step{1.2}
\def\runglen{0.5}
\def\ladderlen{6}
\def\offset{2*\step}

\clip (-\runglen*\ladderlen-2.2,-\runglen*\ladderlen-2.2)
rectangle (\runglen*\ladderlen+1.2,\runglen*\ladderlen+0.7);


\node[smallvx,label=left:$a$] (a) at (-\step,0){};
\node[smallvx,label=right:$b$] (b) at (\step,0){};
\node[smallvx,label=above:$c$] (c) at (0,\step){};
\node[smallvx,label=below:$d$] (d) at (0,-\step){};

\begin{scope}[on background layer]
\draw[hedge,color=dunkelgrau,fill=hellgrau] (a.center) -- (c.center) -- (b.center) -- (d.center) -- cycle;
\end{scope}

\node at (0,0) {$W$};

\begin{scope}[shift={(-\offset,0)}]
\draw[hedge] (0,0.5*\runglen) -- (-\ladderlen*\runglen+\runglen,0.5*\runglen);
\draw[hedge] (0,-0.5*\runglen) -- (-\ladderlen*\runglen+\runglen,-0.5*\runglen);
\foreach \i in {1,...,\ladderlen}{
  \node[smallvx] (C\i) at (-\i*\runglen+\runglen,0.5*\runglen){};
  \node[smallvx] (D\i) at (-\i*\runglen+\runglen,-0.5*\runglen){};
  \draw[hedge] (C\i) -- (D\i);
}
\node at (-0.5*\ladderlen*\runglen+0.5*\runglen,0.5*\runglen+0.5) {$C$};
\end{scope}

\begin{scope}[shift={(0,-\offset)}]
\draw[hedge] (-0.5*\runglen,0) -- (-0.5*\runglen,-\ladderlen*\runglen+\runglen);
\draw[hedge] (0.5*\runglen,0) -- (0.5*\runglen,-\ladderlen*\runglen+\runglen);
\foreach \i in {1,...,\ladderlen}{
  \node[smallvx] (A\i) at (-0.5*\runglen,-\i*\runglen+\runglen){};
  \node[smallvx] (B\i) at (0.5*\runglen,-\i*\runglen+\runglen){};
  \draw[hedge] (A\i) -- (B\i);
}
\node at (0.5*\runglen+0.5,-0.5*\ladderlen*\runglen+0.5*\runglen) {$A$};
\end{scope}

\coordinate (Z) at (45:\offset); 

\def\nstep{1}
\path (Z) ++ (135:\nstep) node[smallvx] (v){};
\path (Z) ++ (135:\nstep+\runglen) node[smallvx] (u){};
\path (Z) ++ (-45:\nstep) node[smallvx] (y){};
\path (Z) ++ (-45:\nstep+\runglen) node[smallvx] (x){};

\pgfmathtruncatemacro{\llen}{\ladderlen-2}
\pgfmathsetmacro{\lstep}{1/(\ladderlen-1)}

\draw[hedge] (v) .. controls +(45:2) and +(45:2) .. (y)
\foreach \p in {1,...,\llen} {node[smallvx,pos=\p*\lstep] (V\p) {}};
\draw[hedge] (u) .. controls +(45:2.5) and +(45:2.5) .. (x)
\foreach \p in {1,...,\llen} {node[smallvx,pos=\p*\lstep] (U\p){}};

\foreach \p in {1,...,\llen} {\draw[hedge] (U\p) -- (V\p);}

\begin{scope}[on background layer]
\draw[bubble] (v.center) .. controls +(45:2) and +(45:2) .. (y.center) -- (x.center) .. controls +(45:2.5) and +(45:2.5) .. (u.center);
\draw[bubble] (A1.center) -- (A\ladderlen.center) -- (B\ladderlen.center) -- (B1.center) -- cycle; 
\draw[bubble] (C1.center) -- (C\ladderlen.center) -- (D\ladderlen.center) -- (D1.center) -- cycle; 
\end{scope}

\path (U4) ++ (0.5,0) node {$B$};

\draw[hedge] (x) -- (y);
\draw[hedge] (u) -- (v);


\draw[ded,out=230,in=0] (x) to (d);
\draw[ded,out=170,in=0] (y) to (c);

\draw[ded,out=-140,in=90] (u) to (a);
\draw[ded,out=-80,in=90] (v) to (b);

\draw[ded,out=10,in=180] (C1) to (c);
\draw[ded,out=-10,in=180] (D1) to (d);

\draw[ded,out=100,in=-90] (A1) to (a);
\draw[ded,out=80,in=-90] (B1) to (b);
\end{tikzpicture}
\caption{Construction in Theorem~\ref{ladders:thm}, although for a shorter ladder}\label{}
\end{figure}
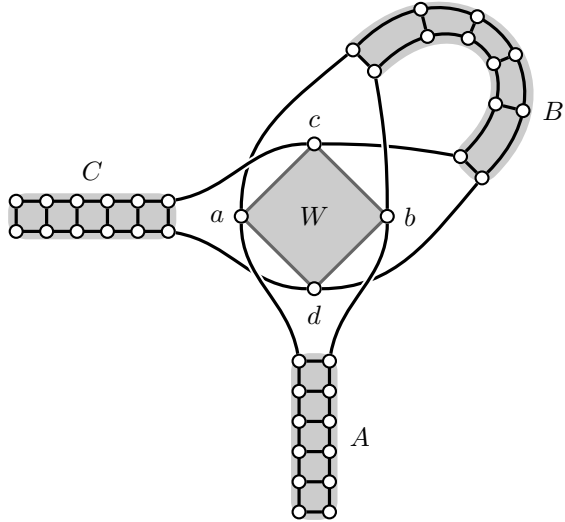

We observe that $\epsilon$ defines an embedding of $V(L)$ in $G$ 
that can easily be extended to an embedding of a subdivision of $L$, if
$u_{24}v_{24}$ and $u_{48}v_{48}$ are mapped to an $(a$--$b,c$--$d)$-linkage in $W$.
Moreover, even if up to $r-2$ edges are deleted in $G$, then 
it is still possible to find such a subdivision of $L$ in the
remaining graph; see here in particular Observation~\ref{wallNoHit}. 
We combine this observation with the following claim
that will take up the rest of the proof:
\begin{equation}\label{edlad}
\emtext{every subdivision of $L$ in $G$ contains an $(a$--$b,c$--$d)$-linkage in $W$.}
\end{equation}
With  Observation~\ref{trees:wallNoTwo}, we deduce from the claim that $G$ cannot contain
two edge-disjoint subdivisions of $L$, which implies, as $r$ can be arbitrarily large, that the subdivisions
of $L$ do not have the edge-Erd\H os-P\'osa property.

To prove~\eqref{edlad},		
 we fix a subdivision $\gamma(L)$ of $L$ in $G$,
		where we treat $\gamma$ as the function that 
		maps every vertex of $L$ to a vertex of $G$
		and every edge $st\in E(L)$ to a $\gamma(s)$--$\gamma(t)$-path in $G$.
		A \emph{rung} of $\gamma(L)$ is a path $\gamma(u_iv_i)$ for some $i\in [71]$.
		
		We first claim that 
		\begin{equation}\label{ladders:rungsInWall}
\emtext{
			$W-T$ contains at most $20$ rungs of $\gamma(L)$.
}
		\end{equation} 
		To prove \eqref{ladders:rungsInWall} consider $\gamma(L)-T$
		and let $L_1,\ldots, L_p$ be the distinct non-empty maximal subdivisions of subladders of $L$ in $\gamma(L)-T$.
		As $|T|=4$, we have $p\le 5$.
		Note that every rung of $\gamma(L)$ is either met by $T$ or is contained in exactly one $L_i$.
		Since $\gamma(L)$ contains $71$ rungs and $T$ meets at most four rungs,
		at least one subladder $L_i$ contains at least six rungs (with room to spare).
		Note that $L_i$ is entirely contained in one component of $G-T$.
		Assume now for a contradiction that $W-T$ contains $21$ rungs of $\gamma(L)$.
		These rungs are contained in certain $L_i$
		and if $W-T$ contains one rung of $L_i$, it contains $L_i$ entirely.
		Since, $W-T$ does not contain a subdivision of a ladder of length at least $6$, by Lemma~\ref{ladders:seventeen},
		every $L_i$ with $L_i\subseteq W-T$ contains at most five rungs.
		If $W-T$ contained $21$ rungs of $\gamma(L)$, it contains at least five subdivisions of subladders $L_i$,
		that is,  $L_1 \cup \ldots \cup L_p$.
		However, we observed that at least one of these subdivisions of ladders has length at least $6$,
		a contradiction to~Lemma~\ref{ladders:seventeen}.
		This proves~\eqref{ladders:rungsInWall}.
		
		Next, we claim that
		\begin{equation}\label{ladders:rungsInEps}
\emtext{
			each of $A,B,C$ contains at most $23$ rungs of $\gamma(L)$.
}
		\end{equation} 
We prove the claim for $A$, the proofs for $B$ and $C$ are almost
the same. 
All inner rungs of $\gamma(L)$ (these are the paths $\gamma(u_iv_i)$ for $2\le i\le 70$) 
		are paths between two vertices of degree $3$ in $\gamma(L)$ and hence degree at least $3$ in $G$.
		The only such vertices in $A$ are those in $R=\{\epsilon(u_i),\epsilon(v_i): 2\le i\le 23\}$.
		As rungs of $\gamma(L)$ are disjoint from each other,
		no two rungs share a vertex of $R$ and hence, $A$ can contain at most $\frac{|R|}{2}=22$ inner rungs of $\gamma(L)$.
As $A$ can contain, additionally, at most one of the rungs $\gamma(u_1v_1)$ 
and $\gamma(u_{71}v_{71})$, we see that $A$ contains at most~$23$ rungs of $\gamma(L)$,
which proves~\eqref{ladders:rungsInEps}.
	
		\medskip

		Set $M=\gamma(L)-T$.
		We claim that
		\begin{equation}\label{ladders:threeComps}
		\begin{minipage}[c]{0.8\textwidth}\em
			$M$ has exactly three components that contain a vertex of $\gamma(U)$,
			namely $\gamma(L)\cap A$, $\gamma(L)\cap B$ 
and $\gamma(L)\cap C$,
			and each of these components contains at least $21$ rungs of $\gamma(L)$.
		\end{minipage}\ignorespacesafterend 
		\end{equation} 
		Let us prove~\eqref{ladders:threeComps}.
		As $M\subseteq G-T$
		it follows from~\eqref{ladders:rungsInWall} and~\eqref{ladders:rungsInEps} 
that no component of $M$ contains $24$ rungs of $\gamma(L)$.
Observe that $M$ has at most four components that contain a vertex of $\gamma(U)$ because $|T|=4$.
				
Suppose first that $M$ contains four components that each contain a vertex 
of  $\gamma(U)$.
		It is not hard to see that this is only possible if there is an $i\in [67]\sm\{1\}$ such that 
\[
\gamma(L-\{u_j,v_j:i\le j \le i+3\})\emtext{ is disjoint from }T.
\]
		\comment{REASON: 
		With four components we have two components at the ends of the ladder, 
		and two middle components. Any middle component that isn't a singleton
		needs four vertices/edges to separate it from the rest. Then on one side
		of the non-singleton middle comp there must be two other components, 
		that are separated from each other AND from the middle component
		by just two vertices/edges. This is impossible. Thus, the two 
		middle components are singletons, which is only possible if $X$ has the desired form.
		}%
As then 
either $i\geq 25$ or $i\leq 71-3-25$ it follows that  there 
is  a subladder $L'$ of $L$ of length~$24$ such that $\gamma(L')$ is disjoint from $T$, 
which implies that $M$ has a component with at least~$24$ rungs --- this is impossible by~\eqref{ladders:rungsInWall} and~\eqref{ladders:rungsInEps}.

Thus, $M$ has at most three components that contain a vertex of $\gamma(U)$.
		Suppose there is such a component $K$ of $M$ that does not contain $21$ rungs of $\gamma(L)$.
		Then, the other two of these components together contain at least $71-20-4=47$ rungs
		and hence, one of them contains at least $24$ rungs.
Again, this is impossible.

Therefore, 
$M$ has three components that contain vertices from $\gamma(U)$, and 
each of these contains at least $21$ rungs of $\gamma(L)$.
By~\eqref{ladders:rungsInWall}, 
$W-T$ does not contain any of these three components,
and thus does not contain any vertex from $\gamma(U)$.
		Thus, the only components of $M$ that can contain a vertex in $\gamma(U)$,
		are $\gamma(L)\cap A$, $\gamma(L)\cap B$ and $\gamma(L)\cap C$.
		This proves~\eqref{ladders:threeComps}.
		
		\medskip
		
		As $\gamma(U)\cap V(W-T)=\emptyset$ by~\eqref{ladders:threeComps}, 
		the only vertices in $G$ that could serve as a vertex in  $\gamma(U)$,
		are the vertices $\epsilon(s)$ for $s\in U$. As $\gamma$ and $\epsilon$ are injective maps,
		we have $|\gamma(U)|=|\epsilon(U)|$. Thus, $\gamma(U)\subseteq \epsilon(U)$ implies that
		\begin{equation}\label{ladders:branchVertices}
			\gamma(U)=\epsilon(U) = (\epsilon(U)\cap (V(A\cup B\cup C)))
\cup\{a,b,c,d\}.
		\end{equation} 
As $A$ is separated from the rest of $G$ by $\{a,b\}$, we deduce
that $A$ must either contain the first~$21$ rungs or the last $21$ rungs
of $\gamma(L)$. By symmetry, we may assume that $A$ contains the rung $\gamma(u_1v_1)$.
As otherwise there would be vertices in $\epsilon(U)\cap V(A)$
that do not lie in $\gamma(U)$, 
contradicting~\eqref{ladders:branchVertices}, it follows that 
all of $L[\{u_j,v_j:j\in[23]\}]$ is mapped to $A$ via $\gamma$,
which in turn implies that $\{\gamma(u_{24},v_{24}\}=\{a,b\}$. 

Where now lies $\gamma(u_{24}v_{24})$? The path is either contained in $W$, 
or it intersects $B$. In the latter case, however, some vertex in $\epsilon(U)\cap V(B)$
lies in the interior of $\gamma(u_{24}v_{24})$, which 
contradicts~\eqref{ladders:branchVertices}. Thus $\gamma(u_{24}v_{24})\subseteq W$.

Arguing in the same way with $C$ we deduce that also that 
$\{\gamma(u_{48}),\gamma(v_{48})\} = \{c,d\}$ and that
$\gamma(u_{48}v_{48})\subseteq W$.
As the two paths $\gamma(u_{24}v_{24})$ and $\gamma(u_{48}v_{48})$ are disjoint
we have found an $(a$--$b,c$--$d)$-linkage in $W$, which proves~\eqref{edlad}.
\end{proof}

\section{Trees of large pathwidth}\label{sec:trees}
	
\newcommand{\SMALL}{12}
\newcommand{\SMALLPLUS}{13}
\newcommand{\SMALLPLUSPLUS}{14}
\newcommand{\SMALLMINUS}{11}

\newcommand{\BORDER}{15}
\newcommand{\BORDERPLUS}{16}
\newcommand{\BORDERMINUS}{14}
\newcommand{\BORDERMINUSMINUS}{13}

\newcommand{\LEVEL}{17}
\newcommand{\LEVELPLUS}{18}
\newcommand{\LEVELPLUSPLUS}{19}

	We expected the expansions of a fixed tree $T$ to have the edge-\EP.
	When the tree is complex enough, however, they do not:

	\begin{theorem}\label{trees:thm}
		If $T$ is a subcubic tree of pathwidth at least \LEVELPLUSPLUS{}, 
		then the family of subdivisions of $T$ does not have the edge-\EP.
	\end{theorem}

Clearly, the theorem implies Theorem~\ref{mainthm}~(ii). 
	We believe that the theorem still holds true for trees of larger maximum degree,
	if instead of subdivisions we consider expansions of $T$ and if
	the pathwidth of $T$ is sufficiently large
	and that this can be shown by using exactly the same construction.
	However, subdivisions and subcubic trees are easier to handle.
	
	To prove Theorem~\ref{trees:thm}
	we construct for any number $r$ (the size of a possible edge hitting set)
	a graph $G$ such that $G$ contains no two edge-disjoint $T$-subdivisions
	but every edge hitting set for $T$-subdivisions contains at least $r$ edges.
	The graph $G$ consists of a condensed wall to which inflations of some parts of the tree $T$ are attached.
	The crucial step lies in proving 
that every subdivision of $T$ in $G$ contains a linkage in the condensed wall.
As, by Observation~\ref{trees:wallNoTwo}, there cannot be two edge-disjoint of these 
we will then have proved the theorem.

We prove Theorem~\ref{trees:thm} in the course of this section.	
	
	\subsection{Binary trees and pathwidth}

	We define a binary tree of height $h\geq 0$ inductively as follows.
	A binary tree of height $0$ is simply the tree with only one vertex, which is also its root.
	A binary tree of height $h>0$ arises from the disjoint union of two binary trees $T_1,T_2$ of height ${h-1}$ and a vertex $r$ (its root) that is joint to the roots of $T_1,T_2$.
	A tree~$T$ is called a \emph{$B_h$-tree} if it is a subdivision of a binary tree $T'$ of height $h$.
	The \emph{root} of $T$ is the branch vertex that corresponds to the root of $T'$.
	We call a tree~$T$ a \emph{$v$-linked $B_h$-tree}
	if $T=T_h\cup P$ for a $B_h$-tree $T_h$ with root $r$ 
	and a $v$--$r$-path~$P$ such that $V(P\cap T_h)=\{r\}$ (we allow that $V(P)=\{r\}$).
	
	Robertson and Seymour \cite{RS83}  were the first 
	to prove that a graph of large pathwidth contains
	a subdivision of a binary tree with large height.
	Marshall and Wood~\cite{MW14} prove an explicit formula if the graph is a tree:
	
	\begin{lemma}[Marshall and Wood \cite{MW14}, restated]\label{trees:coreDepthPathwidth}
		Let $T$ be a tree with at least two vertices.
		Then, $T$ contains a $B_{\pw(T)-1}$-tree.
	\end{lemma}

	For the other direction of Lemma~\ref{trees:coreDepthPathwidth},
	Robertson and Seymour have a tight bound:
	
	\begin{lemma}[Robertson and Seymour~\cite{RS83}] \label{trees:pwBinaryTree}
		A binary tree of height $h$ has pathwidth $\lceil \frac 1 2(h+1) \rceil$.
		A $B_h$-tree has pathwidth $\lceil \frac 1 2(h+1) \rceil$.
	\end{lemma}

For any graph $G$,  
we define $\lev_h(G)$ as 
the set of all vertices $v$ such that  
there are three $v$-linked $B_h$-trees in $G$ that only meet in $v$ and whose root is not $v$.
Note that clearly $\lev_{h+1}(G)\subseteq \lev_{h}(G)$.	
	For every $v\in V(G)$ with $d_G(v)\ge 3$, we define the \emph{level} of $v$ as 
	\[
	\lambda_G(v)=\max\{h: v\in \lev_h(G)\}.
	\]

	For a trees $T'\subseteq T$, 
we define the \emph{weight} $\omega_T(T')$ of $T'$ as $|\lev_{10}(T)\cap V(T')|$.
	For later use we note the following.
	
	\begin{lemma}\label{trees:weightOfBinaryTree}
		The weight of a $B_{k}$-subtree $T'$
of any tree $T$ is at least $2^{k-10}-2$.
	\end{lemma} 
	\begin{proof}
		We may assume that $k\ge 12$, and we may assume $T'$ is a binary tree of height $k$ as by suppressing 
		vertices of degree~$2$ we may only lose vertices of $L_{10}(T)$.
		
		Let $r$ be the root of $T'$, and consider a vertex $w\neq r$
		that has distance at least~$11$ from every leaf of $T'$.
		Then, $w\in L_{10}(T)$.
		The number of such vertices $w$ is 
		\[
		\sum_{i=1}^{k-11} 2^i = 2^{k-10}-2.
		\]
		Therefore, the weight of $T'$ is at least $2^{k-10}-2$.
	\end{proof} 

\subsection{Decomposing the tree}\label{trees:structure}

	For the rest of the section let $T$ be a fixed subcubic tree of pathwidth at least~\LEVELPLUSPLUS{}.
	By Lemma~\ref{trees:coreDepthPathwidth}, $T$ contains a subdivision of a binary tree of height $\LEVELPLUS$
	and hence $\lev_{\LEVEL}(T)\neq \emptyset$.
	To simplify notation,
	we write  $L_i$
	instead of $L_i(T)$, and we also write $\omega(T')$ instead of $\omega_T(T')$
for any subtree $T'$ of $T$.
	Let $U$ denote the set of all vertices of $T$ of degree~$3$.

	Pick a vertex $r$ in $L_{17}$, and then let $\leq_T$ be the 
	usual tree order with root $r$. That is, 
	$u\le_T v$ if and only if the unique $u$--$r$-path in $T$ contains $v$,
	for any two vertices $u,v$ in $T$.
	While the partial order depends on the choice of $r$, we will never use 
	it to compare vertices of $T$ that are contained on the path between two vertices of $L_{17}$.
	Then, however, the 
	actual choice of the root makes no difference. 
	For every vertex $u$ of $T$, 
	let $T_u=T[\{v\in V(T): v\le_T u\}]$ and note that $T_u$ is a tree.

We first state a number of  observations about the structure of $T$ that
we will use throughout this section: 	
\begin{enumerate}[label=(T\arabic{*})]
\item\label{upishuge} The root of $T$ lies in $L_{16}$.
In particular, for every $v\in V(T)$ that is not the root there is a $v$-linked $B_{16}$-tree that 
		meets $T_v$ only in $v$.

\end{enumerate}

We will decompose $T$ into three parts (and two paths connecting these parts), 
a ``large'' part $A$ that contains $L_{\LEVEL}$, 
an intermediate part $C$ and
a ``small'' part~$D$ that does not contain any large binary tree as a subdivision;
	compare Figure~\ref{fig:treeStructure}.
	\begin{figure}[bht]
		\centering
		\begin{tikzpicture}[scale=0.8]
		
		
		\coordinate (v16) at (0,0);
		\coordinate (u16) at (1,0);
		\coordinate (v13) at (7,0);
		\coordinate (u13) at (8,0);

		\fill[rounded corners=10,fill=nochhellergrau] (u16) -- ++(0,-1.5) --  ($(v13)+(0,-1.5)$) -- (v13) -- ++(1,2.5) -- ($(u16)+(-1,2.5)$) -- cycle;

		\draw[hedge] (u16) --
		coordinate[pos=0.12](m1)								
		node[hvertex,fill=hellgrau,pos=0.26,label=below:$u_{\SMALLPLUSPLUS{}}$](u15){} 
		coordinate[pos=0.41](m2)
		coordinate[pos=0.55](m3)
		node[hvertex,fill=hellgrau,pos=0.7, label=below:$u_{\SMALLPLUS{}}$](u14){}
		coordinate[pos=0.85](m4)
		(v13);
		
		\stengelblase{(u16)}{0.7}{90}{1.5}{1}{$B_{15}$}
		\stengelblase{(u15)}{0.5}{90}{1.5}{1}{$B_{14}$}
		\stengelblase{(u14)}{0.5}{90}{1.5}{1}{$B_{13}$}
		\stengelblase{(m1)}{0.3}{90}{0.7}{0.4}{}
		\stengelblase{(m2)}{0.3}{90}{0.7}{0.4}{}
		\stengelblase{(m3)}{0.3}{90}{0.7}{0.4}{}
		\stengelblase{(m4)}{0.3}{90}{0.7}{0.4}{}
		\stengelblase{(v13)}{0.3}{90}{0.7}{0.4}{}
		
		\blase{(v16)}{180}{2.8}{2.4}{}
		\blase{(u13)}{0}{1.7}{1.4}{$D$}
		
		\draw[hedge] (v16)--(u16); 
		\draw[hedge] (v13)--(u13);	
		
		\node[hvertex,minimum size=30] (c18) at (-1.8,0){};
		
		\node at (c18){$L_{\LEVEL}$};
		\node at (-1.1,0.8){$A$};
		
		\node[hvertex,fill=hellgrau,label=above:$v_{\BORDER}$] at (v16){};
		\node[hvertex,fill=hellgrau,label=below:$u_{\BORDER}$] at (u16){};
		\node[hvertex,,fill=hellgrau,label=below:$v_{\SMALL}$] at (v13){};
		\node[hvertex,,fill=hellgrau,label=below:$u_{\SMALL}$] at (u13){};
		
		\node (c) at ($0.5*(u16)+0.5*(v13)+(0,2)$){$C$};
		
		\node (pm) at ($0.5*(u16)+0.5*(v13)+(0,-0.8)$){$P_M$};
		\draw[style=->|] (pm) -- ($(u16)+(0,-0.8)$);
		\draw[style=->|] (pm) -- ($(v13)+(0,-0.8)$);
		\end{tikzpicture}
		\caption{The structure of $T$}
		\label{fig:treeStructure}
	\end{figure}
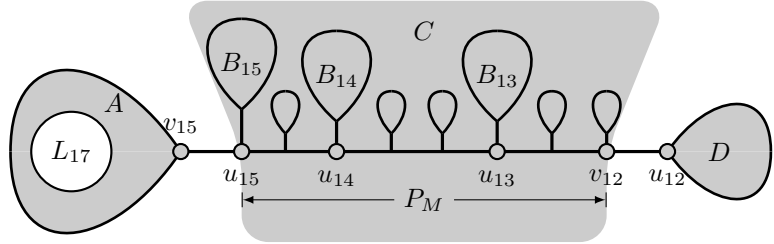

Let $\Smin$ be the set of $\leT$-minimal $\lev_{\BORDER{}}$-vertices.
Among all vertices in $\Smin$, let $u_{\BORDER{}}$ be one such that $\omega(T_{ u_{\BORDER{}}})$ is minimal. 
	Throughout the entire section,
	the weight of $T_{u_{\BORDER{}}}$ is denoted as $\wmin$.
\begin{enumerate}[label=(T\arabic{*}),resume]
\item\label{wmin} $u_{\BORDER}\in \Smin$ and $\min_{s\in \Smin}\omega(T_s)=\omega(T_{u_{15}})=:\wmin$,
where $\Smin$ is the set of vertices that are $\leq_T$-minimal within $L_{15}$.
\end{enumerate}

	As $T$ is subcubic and $T_{u_{\BORDER}}$ contains a $u_{\BORDER}$-linked $B_{\BORDERPLUS}$-tree, 
	the graph $T_{u_{\BORDER{}}}-u_{\BORDER{}}$ consists of two components $T_1,T_2$
each of which contains a $B_{\BORDER{}}$-tree. We pick $T_1,T_2$ such that $\omega(T_1)\geq\omega(T_2)$. 
Choose $u_{14}$ as a vertex in $T_2$ that is $\leq_T$-minimal among the vertices in $L_{14}$. 
Then, for $i=13,12$, choose $u_i$ as a vertex in $T_{u_{i+1}}$ that is $\leq_T$-minimal 
among the vertices in $L_i$. This shows:
	\begin{enumerate}[label=(T\arabic{*}),resume]
		\item\label{ukleaves} For $k\in\{12,\ldots,15\}$,
		the vertex $u_k$ lies in $L_{k}$ and is $\leq_T$-minimal within $L_{k}$.
		In particular, $T_{u_k}$ contains a $B_{k+1}$-tree with root $u_k$.
	\end{enumerate}		
	
For $i\in \{\SMALL{},\BORDER{}\}$, let $v_i$ be the $\leT$-smallest vertex in $U$ with $v_i>_T u_i$.
We define $P_M$, the \emph{main path}, 
as  the path between $u_{15}$ and $v_{12}$ in $T$. Then:
	\begin{enumerate}[label=(T\arabic{*}),resume]
\item\label{ukmain} $P_M$ is the $u_{15}$--$v_{12}$-path in $T$, contains $u_{14}$
and $u_{13}$, and $v_{12}>_T u_{12}$.
\item\label{linorder} $P_M$ is linearly ordered by $\leq_T$.
\item\label{mpath} If $T_1,T_2$ are the two components of $T_{u_{15}}-u_{15}$
such that $T_2$ contains $P_M$, then $\omega(T_1)\geq\omega(T_2)$. 
	\end{enumerate}		

Finally, we define the parts of $T$:
	\begin{enumerate}[label=(T\arabic{*}),resume]
\item\label{ACD} 	The interior vertices of $v_{15}Tu_{15}$ and of $v_{12}Tu_{12}$
have degree~$2$, $D=T_{u_{12}}$, $C=T_{u_{15}}-D-(v_{12}Tu_{12}-v_{12})$
and $A=T-T_{u_{15}}-(v_{15}Tu_{15}-v_{15})$.
	\end{enumerate}		

\subsection{The construction}

Let $r\geq 5$ be an arbitrary positive integer. We construct now a graph $G$ 
that does not admit an edge hitting set of size smaller than $r-2$ and which does not contain 
two edge-disjoint subdivisions of $T$. Roughly speaking, 
$G$ consists of a condensed wall $W$ with terminals $a,b,c,d$
and inflations of $A,C,D$, where we attach the inflation of $A$ to $a$, 
of $D$ to $d$ and of $C$ to $b$ and $c$, respectively;
see Figure~\ref{fig:counterGraph} for illustration.

During this process we define a function $\epsilon$
which in particular maps $U$ to $V(G)$ 
and $U$-paths of $T$ to certain subgraphs of $G$.
It will follow directly by our construction that there is a subdivision $\gamma(T)$ of $T$ in $G$ such that $\gamma(U)=\epsilon(U)$.
Let $\ell=|V(T)|$ and let $G$ be first the empty graph.
\begin{enumerate}[(C1)]
	\item Let $W$ be a condensed wall of size $r$ with terminals $a,b,c,d$; add it to $G$;
	\item for every $u\in U$, add a new vertex $x$ to $G$ and set $\epsilon(u)=x$;
	\item\label{fatedges} for every $U$-path $P$ with endvertices $u,v$ that is
		distinct from $v_{\BORDER{}}Tu_{\BORDER{}}$ and $v_{\SMALL{}}Tu_{\SMALL{}}$,
		add $r$ internally disjoint $\epsilon(u)$--$\epsilon(v)$-paths $P_1,\ldots,P_r$
		of length $\ell$ to $G$ and set $\epsilon(P)=P_1\cup \ldots \cup P_r$;
		\item\label{fatedges2} let $Z_a$ be a set of $r$ internally disjoint $\epsilon(v_{\BORDER{}})$--$a$-paths of length $\ell$,
		$Z_b$ a set of $r$ internally disjoint $b$--$\epsilon(u_{\BORDER{}})$-paths of length $\ell$, 
		$Z_c$ a set of $r$ internally disjoint $c$--$\epsilon(v_{\SMALL{}})$-paths of length $\ell$ and
		$Z_d$ as set of $r$ internally disjoint $d$--$\epsilon(u_{\SMALL{}})$-paths of length $\ell$. 
		Add $Z_a\cup Z_b\cup Z_c\cup Z_d$ to $G$.
	\end{enumerate}
	
	Figure \ref{fig:counterGraph} illustrates the structure of $G$.
	If $R\subseteq V(T)$ is a vertex set, we write $\epsilon(R)$ for $\{\epsilon(v): v\in R \cap U\}$
	and if $T'\subseteq T$ is a subgraph of $T$,
	then let $\epsilon(T')$ be the union of all graphs $\epsilon(P)$ for all $U$-paths in $T'$.
	We did not define $\epsilon(v_{\BORDER{}}Tu_{\BORDER{}})$ and $\epsilon(v_{\SMALL{}}Tu_{\SMALL{}})$.
	Therefore, $\epsilon(T_u)$ is only defined when $T_u$ is edge-disjoint from $v_{\BORDER{}}Tu_{\BORDER{}}$ and $v_{\SMALL{}}Tu_{\SMALL{}}$.
	
	\begin{figure}[bht]
		\centering
		\begin{tikzpicture}[scale=0.8]
		\tikzstyle{tinyvx}=[thick,circle,inner sep=0.cm, minimum size=1.5mm, fill=white, draw=black]

		\def\vstep{0.8}
		\def\hstep{0.4}
		\def\hwidth{11}
		\def\hheight{5}
		
		\def\totalheight{\hheight*\vstep}
		\def\totalwidth{\hwidth*\hstep}
		\pgfmathtruncatemacro{\minustwo}{\hwidth-2}
		\pgfmathtruncatemacro{\minusone}{\hwidth-1}
		
		\foreach \j in {0,...,\hheight} {
			\draw[medge] (0,\j*\vstep) -- (\hwidth*\hstep,\j*\vstep);
			\foreach \i in {0,...,\hwidth} {
				\node[tinyvx] (v\i\j) at (\i*\hstep,\j*\vstep){};
			}
		}
		
		\foreach \j in {1,...,\hheight}{
			\node[hvertex] (z\j) at (0.5*\hwidth*\hstep,\j*\vstep-0.5*\vstep) {};
		}
		\pgfmathtruncatemacro{\plusvone}{\hheight+1}
		
		\node[hvertex,fill=hellgrau,label=10:$c$] (z\plusvone) at (0.5*\totalwidth,\totalheight+0.5*\vstep) {};
		\node[hvertex,fill=hellgrau,label=-10:$d$] (z0) at (0.5*\totalwidth,-0.5*\vstep) {};
		\node (c) at (z\plusvone){};
		\node (d) at (z0){};

		\foreach \j in {1,...,\plusvone}{
			\pgfmathtruncatemacro{\subone}{\j-1}
			\draw[line width=1.3pt,double distance=1.2pt,draw=white,double=black] (z\j) to (z\subone);
			\foreach \i in {1,3,...,\hwidth}{
				\draw[medge] (z\j) to (v\i\subone);
			}
		}
		
		\foreach \j in {0,...,\hheight}{
			\foreach \i in {0,2,...,\hwidth}{
				\draw[medge] (z\j) to (v\i\j);
			}
		}

		\pgfmathtruncatemacro{\minusvone}{\hheight-1}
		\node[hvertex,fill=hellgrau,label=260:$a$] (a) at (-\hstep,0.5*\totalheight) {};
		\foreach \j in {0,...,\hheight} {
			\draw[medge] (a) -- (v0\j);
		}
		
		\node[hvertex,fill=hellgrau,label=280:$b$] (b) at (\totalwidth+\hstep,0.5*\totalheight) {};
		\foreach \j in {0,...,\hheight} {
			\draw[medge] (v\hwidth\j) to (b);
		}
		
		\coordinate (aa) at ($(a)+(-1,0)$){};
		\coordinate (bb) at ($(b)+(1,0)$){};
		\coordinate (cc) at ($(c)+(0,1)$){};
		\coordinate (dd) at ($(d)+(0,-1)$){};
		
		\buendel{(a)}{(aa)}
		\buendel{(b)}{(bb)}
		\buendel{(c)}{(cc)}
		\buendel{(d)}{(dd)}
		
		\node (la) at ($(aa)+(0,-1)$){$Z_a$};
		\draw[->] (la)--($0.5*(aa)+0.5*(a)+(0,-0.2)$);
		
		\blase{(aa)}{180}{2.5}{2}{$\epsilon(A)$}
		\blase{(dd)}{270}{1.5}{1.3}{$\epsilon(D)$}
		
		\draw[hedge] (bb) .. controls ++(3,1) and ++(1,3) ..    
		coordinate[pos=0.18](m1)								
		node[hvertex,fill=hellgrau,pos=0.35,label=180:$\epsilon(u_{\SMALLPLUSPLUS})$](u15){} 
		coordinate[pos=0.45](m2)
		coordinate[pos=0.55](m3)
		node[hvertex,fill=hellgrau,pos=0.65, label=270:$\epsilon(u_{\SMALLPLUS})$](u14){}
		coordinate[pos=0.8](m4)
		(cc);
		
		\stengelblase{(u15)}{0.5}{30}{1.5}{1}{$B_{14}$}
		\stengelblase{(u14)}{0.5}{60}{1.5}{1}{$B_{13}$}
		\stengelblase{(m1)}{0.3}{0}{1}{0.7}{}
		\stengelblase{(m2)}{0.3}{40}{1}{0.7}{}
		\stengelblase{(m3)}{0.3}{50}{1}{0.7}{}
		\stengelblase{(m4)}{0.3}{90}{1}{0.7}{}
		\stengelblase{(bb)}{0.5}{290}{1.5}{1}{$B_{15}$}
		\stengelblase{(cc)}{0.3}{160}{1}{0.7}{}
		
		\node at ($0.85*(cc)+0.85*(bb)$){$\epsilon(C)$};
		
		\node[hvertex,fill=hellgrau] at (aa) {};
		\node[hvertex,fill=hellgrau,label=right:$\epsilon(u_{\BORDER})$] at (bb) {};
		\node[hvertex,fill=hellgrau,label=right:$\epsilon(v_{\SMALL})$] at (cc) {};
		\node[hvertex,fill=hellgrau,label=right:$\epsilon(u_{\SMALL})$] at (dd) {};		
		
		\node (aaa) at ($(aa)+(0,1)$){$\epsilon(v_{\BORDER})$};
		\draw[->] (aaa) -- ($(aa)+(0,0.2)$);
		
		\node at ($0.5*(b)+0.5*(bb)+(0,0.5)$){$Z_b$};
		\node at ($0.5*(c)+0.5*(cc)+(-0.5,0)$){$Z_c$};
		\node at ($0.5*(d)+0.5*(dd)+(-0.5,0)$){$Z_d$};			
		
		\end{tikzpicture}	
		\caption{The counterexample graph $G$}\label{fig:counterGraph}
	\end{figure}
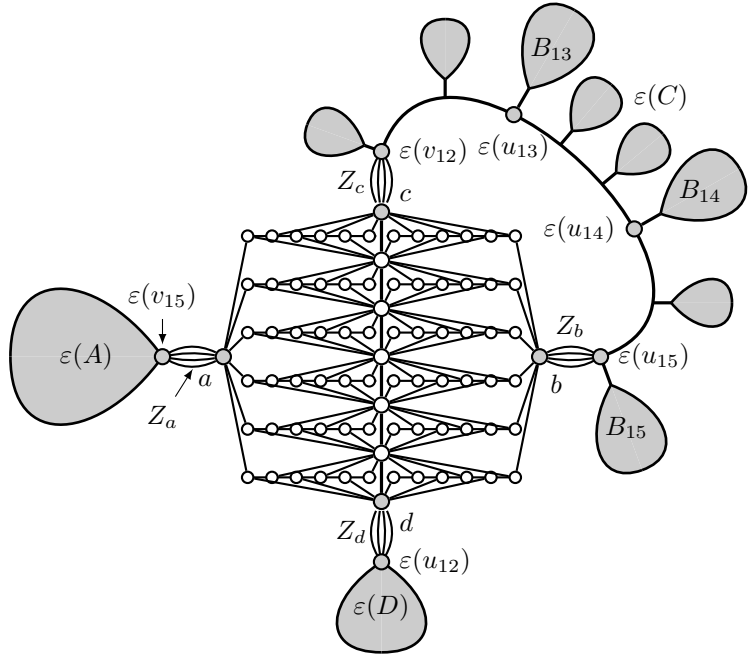
	We often use the subgraph $G'\subseteq G$ which is defined as
	\begin{equation}\label{trees:defGprime}
G' = W \cup \epsilon(C) \cup \epsilon(D) \cup Z_b\cup Z_c\cup Z_d
	\end{equation} 
	and we note that also $G'=G-(\epsilon(A)\cup Z_a - a)$ holds.

We list a number of basic properties of $G$ to which we will appeal later.

	\begin{enumerate}[label=(P\arabic{*})]
		\item \label{treeprops:u16}$\epsilon(\Smin)\cap G'=\{\epsilon(u_{\BORDER})\}$.
		\item \label{treeprops:weight} $|\epsilon(\lev_{10})\cap V(G')| = \wmin$.
		\item \label{treeprops:degree3} If $d_G(x)\ge 3$, then $x\in \epsilon(U)\cup V(W)$.
\item \label{epssep}
Let $v\in U$, and let $T'$ be a component of $T_v-v$ that is disjoint from $P_M\cup\{u_{12}\}$. 
Then $\epsilon(v)$ separates $\epsilon(T')$
from every vertex of degree at least~$3$ in $G-\epsilon(T')$.
\item\label{Usep} Let $v\in U$, and let $U'\subseteq U$ such that $U'$ separates $v$ from $U\setminus \{v\}$ in $T$. 
Then there is a set of $|U'|$ vertices in $\epsilon(U')\cup\{a,b,c,d\}$ 
that separates $\epsilon(v)$ from every other vertex of degree at least~$3$ in $G$.  
	\end{enumerate}

We also have a more complicated property that we formulate as a lemma. 
It, nevertheless, follows immediately from the construction.
\begin{lemma}\label{trees:localStructure}
Let $v\in U\sm V(P_M)$
and let $k$ denote the number of components $T_1,\ldots,T_{k}$ in $T-v$ that intersect $U$.
Then 
\begin{enumerate}[\rm (i)]
\item $G-\epsilon(v)$ has precisely $k$ components; 
\item each of $\epsilon(V(T_1)),\ldots,\epsilon(V(T_{k}))$
 lies in a distinct component of $G-\epsilon(v)$; and
\item for each component $K$ of $G-\epsilon(v)$, there is an $x\in V(K)$ that separates every vertex of degree at least~$3$ of $K$ from $G-K$. 
\end{enumerate}
\end{lemma}

The main lemma we prove in this section:

\begin{lemma}\label{trees:mainLem}
Every subdivision of $T$ in $G$ contains an $(a$--$b,c$--$d)$-linkage
that itself is contained in $W$.
\end{lemma}

Assuming the lemma to be true, we can finish the proof of Theorem~\ref{trees:thm}.

\begin{proof}[Proof of Theorem~\ref{trees:thm}]
We first note that $\epsilon$ is made in such a way that 
there is a subdivision $\tau$ of $T$ in $G$ such that $\tau|_{U}=\epsilon|_{U}$.
Moreover, because of Observation~\ref{wallNoHit} and
because of~\ref{fatedges},\ref{fatedges2},
this 
remains true even
if we delete up to $r-3$ edges from $G$. In particular, any edge hitting set for subdivisions of $T$
will need to have size at least~$r-2$. 

Lemma~\ref{trees:mainLem}, on the other hand, combined with Observation~\ref{trees:wallNoTwo}
shows that $G$ does not contain two edge-disjoint subdivisions of $T$.
\end{proof}

\subsection{Some preparation}

	In the rest of this subsection we prove some lemmas
	that follow from the decomposition of $T$ and the construction of $G$
	and are independent from the $T$-subdivision we will choose later.

	\begin{lemma}\label{trees:epsTrees}
		Let $v\in U$ be a vertex such that $T_v$ is disjoint 
		from $P_M$.
		If $\epsilon(T_v)$ contains an $\epsilon(v)$-linked $B_{\mu}$-tree,
		then also $T_v$ contains a $v$-linked $B_{\mu}$-tree.
	\end{lemma}
\begin{proof}
Clearly, the statement holds for $\mu=0$.
Hence we may assume that $\mu\geq 1$.

Observe that by our construction of $G$, if $T_v-v$ does not intersect $U$, then $\epsilon(T_v)=\epsilon(v)$ and $\mu =0$. 
Hence we may assume that $T_v-v$ intersects $U$.
Suppose $T'$ is a component of $T_v-v$ that intersects $U$ and let $w$ be the $\leq_T$-maximal vertex in $U\cap V(T')$.
Let $K$ be the component of $G-\epsilon(v)$ that contains $\epsilon(T_w)$ given by Lemma~\ref{trees:localStructure}.
Let $x\in V(K)$ be the vertex that separates every vertex of degree at least $3$ of $K$ from $G-K$ also given by Lemma~\ref{trees:localStructure}.
We claim the following:
\begin{equation}\label{downcontain}
\begin{minipage}[c]{0.8\textwidth}\em
If $G[K\cup \{\epsilon(v)\}]$ contains for some $\nu\geq 2$ an $\epsilon(v)$-linked $B_{\nu}$-tree~$F'$,
then $\epsilon(T_{w})$ contains an $\epsilon(w)$-linked $B_\nu$-tree.
\end{minipage}\ignorespacesafterend 
\end{equation} 
As $x=\epsilon(w)$, by \ref{fatedges}, the root of $F'$, which has degree at least~$3$, lies 
in $\epsilon(T_{w})$. 
In addition, $F'\cap\epsilon(T_{w})$ is an $\epsilon(w)$-linked $B_\nu$-tree.
This proves~\eqref{downcontain}.

\medskip

Now let $F$ be an $\epsilon(v)$-linked $B_\mu$-tree in $\epsilon(T_v)$.
We may assume that $\epsilon(v)$ contains no $B_{\mu+1}$-tree.
We already observed above that $V(T_v-v)\cap U \neq \es$.
Hence \ref{fatedges} implies the statement for $\mu=1$.
Thus, we assume $\mu\geq 2$. 
Suppose the root of $F$ is different from $\epsilon(v)$. 
Then it lies in $\epsilon(T_{w})$ for some $\leq_T$-maximal vertex in $V(T_v-v)\cap U$.
Then, by~\eqref{downcontain},
the graph $\epsilon(T_{w})$ contains 
an $\epsilon(w)$-linked $B_\mu$-tree. 
In this case, we replace $v$ by $w$
and proceed with the proof. 

If, on the other hand, the root of $F$ is equal to $\epsilon(v)$, then
$T_v-v$ contains two components that contain a vertex in $U$ (here we exploit Lemma~\ref{trees:localStructure} and $\mu\geq 2$).
This implies the statement for $\mu=2$; so assume from now on that $\mu\geq 3$.
Let $w_1,w_2$ the unique $\leq_T$-maximal vertices in the two components of $T_v-v$ that belong to $U$.
Observe that Lemma~\ref{trees:localStructure} yields two distinct components $K_1,K_2$ of $G-\epsilon(v)$ and vertices $x_1,x_2$ such that 
$\epsilon(T_{w_i})$ lies in $K_i$  and $x_i$ separates every vertex of degree at least 3 of $K_i$ from $G - K_i$.
Now, for each of $i\in [2]$, we can apply~\eqref{downcontain} with $\nu=\mu-1\geq 2$
in order to find an $\epsilon(w_i)$-linked $B_{\mu-1}$-tree in $\epsilon(T_{w_i})$.
With induction on $\mu$ we then find for each $i\in [2]$ a $w_i$-linked $B_{\mu-1}$-tree in $T_{w_i}$,
which finishes the proof. 
\end{proof}
	
\begin{lemma}\label{HDlem}$\,$
\begin{enumerate}[\rm (i)]
\item $W$ does not contain any $B_{10}$-tree.
\item
Let $w\in V(W)$. Every $w$-linked $B_{15}$-tree
that is contained in $G'$ contains a vertex in $\epsilon(C)$.
\end{enumerate}
\end{lemma}
\begin{proof}
(i)
As $W$ has pathwidth at most~$5$, 
by Observation~\ref{wallTreewidth},
but $B_{10}$-trees have pathwidth at least~$6$, by Lemma~\ref{trees:pwBinaryTree}, 
it follows
that $W$ cannot contain any $B_{10}$-tree.

(ii)
Let $F\subseteq G'$ be the  $w$-linked $B_{15}$-tree, and denote 
by $F_1,F_2$ two disjoint  $B_{14}$-trees in $F$ 
that can each be extended to an $w$-linked $B_{14}$-tree in $F$.
If one of $F_1,F_2$
is contained in $\epsilon(D)$, then $\epsilon(D)$ contains 
a $\epsilon(u_{12})$-linked $B_{14}$-tree.
As $D=T_{u_{12}}$ by~\ref{ACD} and as $P_M$ is disjoint from $T_{u_{12}}$
by~\ref{ukmain} and~\ref{linorder},
Lemma~\ref{trees:epsTrees}  implies that $D$ contains a 
$u_{12}$-linked $B_{14}$-tree. Then there exists some vertex $v<_T u_{12}$ 
such that $T_v$ contains a $v$-linked $B_{13}$-tree, which together 
with~\ref{upishuge} implies that 
$v\in L_{12}$. This, however, contradicts
 the $\leq_T$-minimality of $u_{12}$ within $L_{12}$
(by~\ref{ukleaves}).

Thus, at most of one $F_1,F_2$ may meet $\epsilon(D)$; the other, $F_2$ say, 
is disjoint from $\epsilon(D)$. 
By~(i), no $B_{10}$-tree lies completely in $W$, which means that $F_2$, 
as a $B_{14}$-tree, meets $\epsilon(C)$.  
\end{proof}

We stick to the following convention for the rest of this section:
	We denote by $s,t,u,v,w$ always vertices in $T$
	and $x,y$ are vertices in $G$.
	Trees in $G$ are called $F$ (with some index) and subtrees of $T$ have name $T$ (with some index).

	\begin{lemma}\label{trees:lambdaInvariance}
		Suppose $v\in U$ and $\lambda_T(v)\le \BORDER{}$.
		Then $\lambda_G( \epsilon(v))=\lambda_T(v)$ if $\lambda_T(v)\geq 2$;
and  $\lambda_G( \epsilon(v))\in \{0,1\}$ if $\lambda_T(v)\leq 1$.
	\end{lemma}	
	\begin{proof}
The statement for $\lambda_T(v)=1$ follows easily by our construction.
Hence we may assume from now that $\lambda_T(v)\geq 2$.
First, we always have $\lambda_G(\epsilon(v))\ge \lambda_T(v)$,
as we can find $T$ as a subdivision in $G$ such that $\epsilon(v)$ is the corresponding branch vertex of $v$.
	
Next, we set $\mu:=\lambda_T(v)$ and prove that 
		$\lambda_G(\epsilon(v))\leq \mu$.
If $\mu\leq 1$, then $v$ can be separated from $U\setminus \{v\}$ by at most two vertices in $T$.
With~\ref{Usep} it follows that also $\lambda_G(\epsilon(v))\leq 1$. 
Thus, we assume that $\mu\geq 2$. 
		Denote by $w_1,w_2$ the $\leq_T$-largest two vertices in $U$ with $v>w_1$ and $v>w_2$.
Note that $v$ is not the root of $T$ as $\lambda_T(v)\leq 15$ but the root 
lies in $L_{16}$ by~\ref{upishuge}. Thus, again by~\ref{upishuge} and by~$\mu\leq 15$, 
it follows that for one $i\in\{1,2\}$
		the tree $T_{w_i}$ does not contain any $w_i$-linked $B_{\mu+1}$-tree.
		We assume that this is the case for $i=1$. 
		
Assume first that $T_{w_1}$ is disjoint from $P_M$ and $u_{12}$,
and suppose that $\lambda_G(\epsilon(v))\geq\mu+1$.
Then,~\ref{Usep} and a look at Figure~\ref{fig:counterGraph} shows that 
there is a $\epsilon(v)$-linked $B_{\mu+1}$-tree $F$ in $G$ such that the path 
between $\epsilon(v)$ and the root of $F$ passes through~$\epsilon(w_1)$.
It follows from~\ref{epssep} that $\epsilon(T_{w_1})$ contains an $\epsilon(w_1)$-linked
$B_{\mu+1}$-tree. Then, however, an application of 
Lemma~\ref{trees:epsTrees} yields a $w_1$-linked $B_{\mu+1}$-tree in $T_{w_1}$, which we had excluded.

Therefore, we assume from now on that
		$T_{w_1}$ intersects $P_M$ or contains $u_{12}$. 
		By~\ref{ukleaves}, there is a $u_{\BORDER}$-linked $B_{16}$-tree in $T_{u_{\BORDER}}$,
		which means because of $\mu<16$ that $u_{\BORDER}$ cannot lie in $T_{w_1}$.
Then, as $T_{w_1}$ intersects $P_M$ or contains $u_{12}$, we 
see that $v\in V(P_M)$, and as a consequence of~\ref{ukmain}
 that
		\begin{equation}\label{rosencrantz}
		u_{12}\in V(T_{w_1}).
		\end{equation}
We deduce that $T_{u_{12}}\subseteq T_{w_1}$, 
which by~\ref{ukleaves} implies that $T_{w_1}$ contains a $w_1$-linked $B_{13}$-tree,
i.e.\ that
$\mu\geq 13$. 
		
		As there is no $w_1$-linked
		$B_{\mu+1}$-tree in $T_{w_1}$ it follows that no vertex in $T_{w_1}$ 
lies in $L_{\mu}$. Since $v\in L_\mu\cap V(P_M)$ we see 
with~\ref{linorder} that $v$ is $\leq_T$-minimal
		within $L_\mu\cap V(P_M)$. As $\mu\in\{13,14,15\}$ it follows  
		from~\ref{ukleaves} and~\ref{ukmain} that $v=u_{\mu}$.
		Then, by~\ref{ukleaves}, $v$ is $\leq_T$-minimal within $L_\mu$, which 
		in particular together with~\ref{upishuge}
		implies that $T_{w_2}$ cannot contain any $w_2$-linked $B_{\mu+1}$-tree.
		
		We can now exchange the roles of $w_1$ and $w_2$ and deduce in the same
		way that either $\lambda_G(\epsilon(v))\leq \mu$ or that
		$u_{12}\in V(T_{w_2})$. As $u_{12}$ can lie in only one of $T_{w_1}$
		and $T_{w_2}$, we must have $\lambda_G(\epsilon(v))\leq \mu$.
	\end{proof}

	\subsection{$L_{10}$-vertices of $\gamma(T)$}

	We split the proof of Lemma \ref{trees:mainLem}
	in several smaller pieces.
	In addition that we fixed $T$ already before,
	we now also fix some $T$-subdivision $\gamma(T)$ in $G$, 
	where we think of $\gamma$ as the map that maps every vertex $t\in V(T)$ to a vertex $\gamma(t)$ in $G$
	and every edge $st\in E(T)$ to a $\gamma(s)$--$\gamma(t)$-path in $G$
	such that $\gamma(st)\cap \gamma(uv)=\emptyset$ if $\{s,t\}\cap \{u,v\}=\emptyset$
	and  $\gamma(st)\cap \gamma(uv)=\{\gamma(s)\}$ if the edges $st$ and $uv$ share the vertex $s$.
	
	For $F\subseteq G$, we define $\omega_{\gamma}(F)=|\gamma(L_{10})\cap V(F)|$.
	If $F$ is a subgraph of $G$, the reader may think of the number $|\epsilon(L_{10})\cap V(F)|$
	as the \emph{capacity} of $F$
	while $\omega_\gamma(F)=|\gamma(L_{10})\cap V(F)|$ may be called the \emph{utilization} of $F$.
	A key step in the proof of Lemma \ref{trees:mainLem} is to
	show that the utilization and the capacity of any subgraph of $G$ differ at most by 2.
	
	\begin{lemma}\label{trees:atMostTwo}
		The set $\gamma(\lev_{10}) \setminus \epsilon(\lev_{10})$ is contained in $W$
		and $|\gamma(\lev_{10}) \setminus \epsilon(\lev_{10})|=
		|\epsilon(\lev_{10}) \setminus \gamma(\lev_{10}) |\le 2$.
		In particular, for any subgraph $G_1\subseteq G$,
		it holds  that
		\[
		|\epsilon(L_{10})\cap V(G_1)| +2 \ge |\gamma(L_{10})\cap V(G_1)| \ge |\epsilon(L_{10})\cap V(G_1)| -2.
		\]
	\end{lemma}
	\begin{proof}
		Since $\epsilon$ as well as $\gamma$ map injectively a single vertex of $T$ to a single vertex of $G$,
		it holds that $|\gamma(\lev_{10})|=|\lev_{10}|=|\epsilon(\lev_{10})|$ and consequently
		$|\gamma(\lev_{10}) \setminus \epsilon(\lev_{10})|=|\epsilon(\lev_{10}) \setminus \gamma(\lev_{10}) |$.
		We will show now that $\gamma(\lev_{10}) \setminus \epsilon(\lev_{10})$
		is contained in $W$ and consists of at most two vertices.
		
		Assume first that there is an $x\in \gamma(\lev_{10}) \setminus (\epsilon(\lev_{10})\cup V(W))$.
		Since $x\in \gamma(\lev_{10})$,
		we observe that $\lambda_G(x)\ge 10$ and $d_G(x)\ge 3$.
		As it is not contained in $\epsilon(\lev_{10})$ but has degree at least 3, 
		there is a vertex $v\in \lev_i \setminus \lev_{10}$ for some $i\in \{0,\ldots, 9\}$
		with $x=\epsilon(v)$.
		As $\lambda_T(v)\le 9$, Lemma~\ref{trees:lambdaInvariance} shows $\lambda_G(x)=\lambda_G(\epsilon(v))=\lambda_T(v)\le 9$,
		which is a contradiction to $\lambda_G(x)\ge 10$. Hence $\gamma(\lev_{10}) \setminus \epsilon(\lev_{10})\subseteq V(W)$.
		
		Assume now that $|\gamma(L_{10})\cap V(W)|\ge 3$. 
		No $B_{10}$-tree is completely contained in $W$, by Lemma~\ref{HDlem}~(i).
		Hence, for every vertex $x\in \gamma(\lev_{10})\cap V(W)$,
		there are three $x$-linked $B_{10}$-trees that only meet in $x$
		such that each of these trees contains a vertex of $G-W$
		and therefore contains a distinct terminal of $W$.
		Now, three vertices in $\gamma(\lev_{10})\cap V(W)$
		give rise to at least five disjoint paths from $\gamma(\lev_{10})\cap V(W)$ to $G-W$,
		which contradicts Menger's theorem because the four vertices $\{a,b,c,d\}$
		separate $\gamma(\lev_{10})\cap V(W)$ from $G-W$.
		This proves the lemma.
	\end{proof}

\begin{lemma}\label{ADlem}$\,$
\begin{enumerate}[\rm (i)]
\item $\gamma(T)$ meets $\epsilon(A)$ and $\epsilon(D)$. In particular, 
$\gamma(T)$ contains $a$ and $d$.
\item $\omega(D)\geq 6$ and $\omega_\gamma(\epsilon(D))\geq 4$.
\end{enumerate}
\end{lemma}
\begin{proof}
By~\ref{upishuge},~\ref{ukleaves} and~\ref{ACD},
both trees $A$ and $D$ contain a $B_{13}$-tree, and thus have, 
by Lemma~\ref{trees:weightOfBinaryTree}, a weight of at least~$2^3-2\geq 6$.
Lemma~\ref{trees:atMostTwo} implies that $\omega_\gamma(\epsilon(A))\geq 6-2=4$
and $\omega_\gamma(\epsilon(D))\geq 4$. Therefore, $\gamma(T)$ meets
both of $\epsilon(A)$ and $\epsilon(D)$. 
\end{proof}

	\subsection{Embedding outside $A$}
	\label{trees:subtreeMapping}
	
	Recall that $\Smin\subseteq V(T)$ is the set
	of $\leT$-minimal vertices within $L_{\BORDER{}}$; see~\ref{wmin}.
	A crucial step in the proof of Lemma~\ref{trees:mainLem}
	is to locate $T_s$ for $s\in \Smin$.
	The following lemma provides a handy criterion
	to exclude parts of $G$.
	
\begin{lemma}\label{trees:subtree2}
Let $G_1\subseteq G$ be separated by a  vertex $x\in V(G_1)$
from $G-G_1$, and let there be $y\in V(G-G_1)$ that separates all 
vertices of degree at least~$3$ in $G-G_1$ from $G_1$. 
Let $v\in L_{13}$, and assume that $\gamma(v)\in V(G_1)$ and that
$|\epsilon(L_{10})\cap V(G-G_1)|\ge \omega(T_v)-2$.
		Then, $\gamma(T_v)\subseteq G_1$.
	\end{lemma}
	\begin{proof}
As $v\in L_{13}$, the tree $\gamma(T_v)$ splits into two edge-disjoint trees $F_1,F_2$
such that each is a $\gamma(v)$-linked $B_{13}$-tree.
By looking at pre-images of $F_1,F_2$ and by applying 
Lemma~\ref{trees:weightOfBinaryTree}, we see that 
\(
\omega(F_i)=|\gamma(L_{10})\cap V(F_i)|\geq 6
\)
for each $i\in [2]$.

Suppose that one of $F_1,F_2$ meets $G-G_1$. 
Then $\gamma(T)\cap (G-G_1)\subseteq \gamma(T_v)=F_1\cup F_2$
as $x$ separates $G_1$ from $G-G_1$ and as $\gamma(T_v)$ meets $G_1$, by assumption.
 
As $\omega(T_v)\geq \omega(F_1)\geq 6$ and using the assumptions of the lemma, it follows that
$|\epsilon(L_{10})\cap V(G-G_1)|\ge \omega(T_v)-2\geq 4$,
which means, by  Lemma~\ref{trees:atMostTwo}, that 
one of $F_1,F_2$, say $F_1$, must meet a vertex in $\epsilon(L_{10})\cap V(G-G_1)$.
Then, $F_1$ in particular contains a vertex of degree at least~$3$ 
in $G-G_1$, which implies that $y\in V(F_1)$, and then that $F_2$ is disjoint 
from $\gamma(L_{10})\cap V(G-G_1)$, and in turn that 
$\gamma(L_{10})\cap V(G-G_1)\subseteq V(F_1-x)$.
We apply Lemma~\ref{trees:atMostTwo} again and obtain
\begin{align*}
|\gamma(L_{10})\cap V(F_1-x)|& \geq
|\gamma(L_{10})\cap V(G-G_1)|\\
& \geq |\epsilon(L_{10})\cap V(G-G_1)|-2\geq 
\omega(T_v)-4.
\end{align*}
We continue
\begin{align*}
\omega(T_v)& = |\gamma(L_{10})\cap V(\gamma(T_v))| 
= |\gamma(L_{10})\cap V(F_1-x)| + |\gamma(L_{10})\cap V(F_2)|\\
& \geq \omega(T_v)-4 + 6 > \omega(T_v),
\end{align*}
		which is impossible.
		Therefore,  $F_1\cup F_2$ does not meet $G-G_1$ and hence, $\gamma(T_v)\subseteq G_1$.
	\end{proof}

	Lemma~\ref{trees:rootMapping}
	further elaborates on the location of $\gamma(T_s)$
	for $s\in \Smin$ and $\omega(T_s)\le \wmin +2$.

	\begin{lemma}\label{trees:rootMapping}
		Let $s\in \Smin$ and $\omega(T_s)\le \wmin+2$.
		Then, $\gamma(s)\in \epsilon(\Smin) \cup V(W)$.
		Moreover, if $\gamma(s)=\epsilon(t)$ for some $t\in \Smin\sm \{u_{\BORDER}\}$, 
		then $\gamma(T_s)\subseteq \epsilon(T_{t})$.
	\end{lemma}
	\begin{proof}
		If $\gamma(s)\in V(W)$, then there is nothing left to prove. 
		Thus, let $\gamma(s)\notin V(W)$, which implies that there is some $s'\in V(T)$
		with $\gamma(s)=\epsilon(s')$ (as $\gamma(s)$ has degree~$3$ in $\gamma(T)$, by \ref{treeprops:degree3}).
		
		Next, observe that we must have $\lambda_G(\epsilon(s'))=
		\lambda_G(\gamma(s))\geq 15$ because 
		$s\in L_{15}$
		and $\gamma(s)$ is the corresponding branch vertex of $s$ in $\gamma(T)$.
		By Lemma~\ref{trees:lambdaInvariance}, this implies that 
		\begin{equation}\label{guildenstern}
		s'\in L_{15}.
		\end{equation}
		
		Next, we prove that
		\begin{equation}\label{trees:notInside}
		s'\in \Smin.
		\end{equation} 
		Suppose not. Let $T_1,T_2,T_3$ be the three components of $T-s'$,
where we assume that $T_2,T_3\subseteq T_{s'}$. 
		
By~\ref{upishuge}, there is a vertex in $L_{15}$ contained
		in $T-T_{s'}=T_1$ that is incomparable with $s'$, and then also a vertex $s_1\in \Smin$
		such that $T_{s_1}\subseteq T_1$. 
		As $s'$ is not $\leq_T$-minimal within $L_{15}$, one of $T_{2}$ and $T_{3}$
		contains a vertex from $\Smin$, say that $s_2\in \Smin\cap V(T_{2})$.
		Then 
		\begin{equation}\label{polonius}
		\omega(T_1),\omega(T_2)\geq\wmin,
		\end{equation}
		by~\ref{wmin}.
		
		Because of~\ref{ukleaves},~\ref{ukmain} and~\ref{linorder} the vertex $u_{15}$
		is the only vertex in $P_M$ that lies in $L_{15}$. 
		Since $s'\neq u_{\BORDER}$ as $s'\notin \Smin$
		by~\ref{ukleaves}, we deduce from~\eqref{guildenstern} that $s'\notin V(P_M)$.
By Lemma~\ref{trees:localStructure}, the graph $G-\epsilon(s')$ has precisely three
components $K_1,K_2,K_3$, it holds that $K_i$ contains $\epsilon(V(T_i))$,
and there is $x_i\in V(K_i)$ that separates all vertices of degree~$3$
or more in $K_i$ from $G-K_i$ for $i\in [3]$.
Note that by~\eqref{polonius} 
\[
|\epsilon(L_{10})\cap V(K_i)|\geq \wmin \geq \omega(T_s)-2 \emtext{ for each }i\in[2].
\]		

Applying Lemma~\ref{trees:subtree2} with $G_1=G-K_1$, $x=\epsilon(s')$, $y=x_1$,
we first see that $\gamma(T_s)\subseteq G-K_1$,
and then, with a second application, that also $\gamma(T_s)\subseteq G-K_2$. It follows
that $\gamma(T_s)\subseteq G-K_1-K_2=G[K_3+\epsilon(s')]$. 

On the other hand, the tree $\gamma(T_s)$ contains two $\epsilon(s')$-linked $B_{\BORDER}$-trees that only meet in $\epsilon(s')=\gamma(s)$.
But as both of them are contained in $G[K_3+\epsilon(s')]$, 
they have to contain $x_3\neq \epsilon(s')$,
		which is impossible.
		This proves~\eqref{trees:notInside},
		which then implies that $\gamma(s)\in V(W) \cup\epsilon(\Smin)$.
		
	\medskip
		Finally, we claim that
		\begin{equation}
\emtext{
		if $t\in \Smin \sm \{u_{\BORDER}\}$ and $\gamma(s)=\epsilon(t)$, then $\gamma(T_s)\subseteq \epsilon(T_t)$.
}
		\end{equation} 
Observe that $\epsilon(T_t)$ is defined as $T_t$ is disjoint from $P_M$;
the latter follows from~\ref{ukleaves} and~\ref{linorder}.
Put $G_1=\epsilon(T_t)$, and note that $G-G_1$ contains $\epsilon(u_{15})$.
It follows from~\ref{ukleaves} that $G-G_1$ also contains a
subdivision of $T_{u_{\BORDER}}$. By~\ref{wmin}, this implies
$|\epsilon(L_{10})\cap V(G-G_1)|\ge \wmin$.
Moreover,  Lemma~\ref{trees:localStructure} yields a vertex~$y$ such that 
Lemma~\ref{trees:subtree2} becomes applicable to $G_1$ with $x=\epsilon(t)$.
We obtain $\gamma(T_s)\subseteq \epsilon(T_t)$.
	\end{proof}

	For a vertex $u\in U$, we define the \emph{signature} $\sigma(T_u)\in \N^2$ of the tree $T_u$ as follows:
	Let $v,w$ be the neighbours of $u$ with $v,w\leT u$
	and suppose $\omega(T_v)\ge \omega(T_w)$
	Then, $\sigma(T_u)=(\omega(T_v),\omega(T_w))$.
	We write $\sigma(T_1)\ge \sigma(T_2)$ if the inequality holds componentwise.
	
	The next lemma shows that if 
	for some $s,t\in \Smin$ the tree $T_s$ is mapped into $\epsilon(T_t)$, 
	the intended space for $T_t$, then the signature of $T_t$
	is at least as large as the signature of $T_s$.
	
	\begin{lemma}\label{trees:mapSignature}
		Let $s_1,s_2\in \Smin$, set $Q_1=T_{s_1}$ and $Q_2=T_{s_2}$, and assume  that
		$Q_2\subseteq A$.
		If $\gamma(Q_1)\subseteq \epsilon(Q_2)$,
		then $\sigma(Q_1)\leq \sigma(Q_2)$.
	\end{lemma}
	\begin{proof}
By Lemma~\ref{trees:atMostTwo}, 
we have $\epsilon(A-L_{10})\cap \gamma(L_{10})=\emptyset$.
		Therefore, if $T_1,T_2$ are two subtrees of $T$, then we clearly have
		\begin{equation}\label{trees:weightEmbedding}
		\begin{minipage}[c]{0.8\textwidth}\em
		$T_2\subseteq A$, $\gamma(T_1)\subseteq \epsilon(T_2) \Longrightarrow \omega(T_1)\le \omega(T_2)$.
		\end{minipage}\ignorespacesafterend 
		\end{equation}

		For each $i\in [2]$, let $\sigma(Q_i)=(\alpha_i,\beta_i)$,
		and let  $Q_{i,1}$ and $Q_{i,2}$ be the two components in $Q_i-s_i$ such 
		that $\omega(Q_{i,1})=\alpha_i$ and $\omega(Q_{i,2})=\beta_i$. 
		Note that, because of $s_i\in \Smin$ each of $Q_{i,1}$ and $Q_{i,2}$ contains
		a $B_{15}$-tree that is linked to the respective root.

		Assume for a contradiction that $(\alpha_2,\beta_2)\ngeq (\alpha_1,\beta_1)$.
		Suppose first that $\beta_2<\beta_1$.
		Since $\alpha_1\ge \beta_1>\beta_2$,
		neither $\gamma(Q_{1,1})$ nor $\gamma(Q_{1,2})$ can be contained in $\epsilon(Q_{2,2})$, by \eqref{trees:weightEmbedding}.
		This implies that both trees $\gamma(Q_{1,1})$ and $\gamma(Q_{1,2})$
		contain edges of $\epsilon(Q_{2,1})$.
		As $\gamma(Q_{1,1})$ and $\gamma(Q_{1,2})$ are disjoint,
		this is, by 
Lemma~\ref{trees:localStructure}, only possible if the whole tree $\gamma(Q_1)$ is contained in $\epsilon(Q_{2,1})$.
		However, this implies that $Q_{2,1}$ contains a $B_{\BORDERPLUS{}}$-tree,
		which together with~\ref{upishuge}, shows that the root of $Q_{2,1}$ 
		lies in $L_{\BORDER}$. This, however, contradicts $s_2\in \Smin$.
		
		Now suppose that $\beta_2\ge \beta_1$ but $\alpha_2<\alpha_1$.
		So we have $\alpha_1 > \alpha_2\ge \beta_2 \ge  \beta_1$.
		Hence, by \eqref{trees:weightEmbedding}, 
		the tree $\gamma(Q_{1,1})$ is neither contained in $\epsilon(Q_{2,1})$ nor in $\epsilon(Q_{2,2})$
		and therefore contains edges of both trees $\epsilon(Q_{2,1})$ and $\epsilon(Q_{2,2})$.
		But then, 
by Lemma~\ref{trees:localStructure},
there is no place for $\gamma(Q_{1,2})$ as it is disjoint from $\gamma(Q_{1,1})$
		but also contained in $\epsilon(Q_2)$.
		This is the final contradiction.
	\end{proof}
	
	Recall from~\ref{treeprops:u16}
	that $u_{\BORDER}\in \Smin$, $\omega(T_{u_{\BORDER}})=\wmin$ and
	$\epsilon(V(T_{u_{\BORDER}}))\subseteq V(G')$
	with $G'=G-(\epsilon(A)\cup Z_a - \{a\})$.
	So, $\epsilon$ maps $T_{u_{\BORDER}}$ to $G'$,
	but $\gamma$ may map it to some other part of $G$.
	Which part of $T$ is then mapped to $G'$ by $\gamma$?
	The next lemma shows that this is
	a $\leT$-minimal tree that contains a $B_{\BORDERPLUS}$-tree.

	\begin{lemma} \label{trees:treeOutsideA}
		There is a vertex $s\in \Smin$ such that  
		$\gamma(T_s)\subseteq G'$ and $\sigma(T_s)\ge \sigma(T_{u_{\BORDER{}}})$.
	\end{lemma}
	\begin{proof}
		Let $s_1=u_{\BORDER}$.
		Starting from $s_1$, we construct a sequence $s_1,\ldots, s_h$ 
		of vertices such that $\gamma(s_i)=\epsilon(s_{i+1})$ for every $i\in [h-1]$
		such that the following properties hold
		\begin{enumerate}[\rm (i)]
			\item $s_i\in \Smin$ for every $i\in [h]$;
			\item $\omega(T_{s_i})\le \wmin+2$ for every $i\in [h]$; 
			\item $\sigma(T_{s_{i}})\ge \sigma(T_{u_{\BORDER}})$ for every $i\in [h]$;
			\item $\gamma(s_i)\in \epsilon(A)$ for every $i\in [h-1]$ and $\gamma(s_h)\in V(G')$.
		\end{enumerate}
		Then, $s_h$ will serve as the vertex $s$ in the statement of the lemma.
		To find this sequence, we first check that $s_1$ satisfies all properties (i)--(iv)
		and prove that the properties are maintained from each $s_i$ to its successor.
		To avoid double subscripts, we write $T^{(i)}$ instead of $T_{s_{i}}$ for every~$i$.
		
		By \ref{wmin}, $s_1=u_{\BORDER}\in \Smin$ and $\omega(T^{(1)})=\wmin$ and trivially, $\sigma(T^{(1)})\ge \sigma(T_{u_{\BORDER}})$.
		Since $\omega(T^{(1)})\le \wmin$, Lemma~\ref{trees:rootMapping} shows that $\gamma(s_1)\in W\cup \epsilon(\Smin)$.
		If $\gamma(s_1)\in V(G')$, we set $h=1$ and stop this process.
		Hence, we may assume that $\gamma(s_1)\in \epsilon(\Smin\sm \{u_{\BORDER}\})\subseteq \epsilon(A)$
		and therefore, $s_1$ satisfies all properties (i)--(iv).
		
		Now let $i\in \N$ be a number such that $s_1,\ldots, s_i$ are already constructed,
		each $s_j$ for $j\le i$ satisfies (i)--(iv) and $\gamma(s_i)\notin V(G')$.
		We prove now that $s_{i+1}$ is well-defined via $\gamma(s_i)=\epsilon(s_{i+1})$ and that it satisfies (i)--(iv), as well.
		As $s_i\in \Smin$ by (i) and $\omega(T^{(i)})\le \wmin+2$ by (ii), 
		Lemma~\ref{trees:rootMapping} implies that $\gamma(s_i)\in V(W)\cup \epsilon(\Smin)$.
		Since $\gamma(s_i)\notin V(G')$, we have $s_i\in \epsilon(\Smin\sm \{u_{\BORDER}\})$.
		So, the vertex $s_{i+1}\in \Smin\sm \{u_{\BORDER}\}$ is well-defined via $\gamma(s_i)=\epsilon(s_{i+1})$ and therefore satisfies~(i).
		
		Since $\omega(T^{(i)})\le \wmin+2$, by (ii),
		Lemma~\ref{trees:rootMapping} implies that $\gamma(T^{(i)})\subseteq \epsilon(T^{(i+1)})$.
		This in turn implies by Lemma~\ref{trees:mapSignature} that $\sigma(T^{(i+1)})\ge \sigma(T^{(i)})$
		and as $\sigma(T^{(i)})\ge \sigma(T_{u_{\BORDER}})$, by (iii),
		we also have $\sigma(T^{(i+1)})\ge \sigma(T_{u_{\BORDER}})$.
		Thus, $s_{i+1}$ satisfies (iii).

		To prove (ii), assume for a contradiction that $\omega(T^{(i+1)})>\wmin+2$.
		Then there are three distinct vertices $t_1,t_2,t_3\in V(T)$ and
		numbers $1\le j_1\leq j_2\leq j_3 \le {h-1}$
		such that for every $i\in  [3]$, 
		we have $t_i\in \lev_{10}\cap V(T^{(j_i+1)})$ 
		but $\epsilon(t_i)\notin \gamma(L_{10})\cap V(\gamma(T^{(j_i)}))$. 
		This implies that $\epsilon(t_i)\notin \gamma(L_{10})$
		for each $i\in [3]$, which is a contradiction to Lemma~\ref{trees:atMostTwo}.
		Hence, $\omega(T^{(i+1)})\le \wmin+2$, and therefore $s_{i+1}$ satisfies (ii).
		
		Now, if $\gamma(s_{i+1})\in V(G')$, set $h=i+1$ and stop,
		otherwise we have $\gamma(s_{i+1})\in \epsilon(\Smin\sm\{u_{\BORDER}\}) \subseteq \epsilon(A)$ 
		and therefore, $s_{i+1}$ also satisfies (iv).
		
		\medskip
		
		It is not difficult to see that this process terminates.
		Indeed, assume for a contradiction that it does not.
		Since there are only finitely many vertices in $\Smin$,
		there must be indices $1\le i < j \le |\Smin|+1$ such $s_i=s_j$.
		Among all such pairs of indices choose $(i,j)$ such that $|j-i|$ is minimal and subject to that $i$ is minimal.
		Then, by this minimality, we have $s_{i-1}\neq s_{j-1}$ 
		but on the other hand $\gamma(s_{i-1})=\epsilon(s_i)=\epsilon(s_j)=\gamma(s_{j-1})$.
		This is a contradiction as $\gamma$ maps $V(T)$ injectively
		into $V(G)$.

		Therefore, the process terminates with a vertex $s_h\in \Smin$ such that 
		$\gamma(s_h)\in V(G')$ and $\sigma(T^{(h)})\ge \sigma(T_{u_{\BORDER}})$.
		Set $s=s_h$.
		Let $G_1=G'$, $x=a$ and $y=\epsilon(v_{\BORDER})$.
		We have $|\epsilon(L_{10})\cap V(G-G_1)| \ge \wmin$ as $\epsilon(A)$ contains a tree $T_{s'}$ for some $s'\in \Smin\sm\{u_{\BORDER}\}$,
		and this tree has weight $\wmin$ by \ref{wmin}.
		Since $\gamma(s)\in V(G_1)$, Lemma~\ref{trees:subtree2} implies that $\gamma(T_s)\subseteq G_1=G'$.
		Hence, $s$ satisfies the statement of the lemma.
	\end{proof}
	
	\subsection{Finding a linkage}
	\label{trees:findLinkage}
	We finally come to the proof of Lemma \ref{trees:mainLem}.
	By Lemma~\ref{trees:treeOutsideA}, there is a vertex $s^*\in \Smin$ 
	such that 
	 $\gamma(T_{s^*})\subseteq G'$
	and $\sigma(T_{s^*}) \ge \sigma(T_{u_{\BORDER{}}})$.

	Consider the unique $a$--$\gamma(s^*)$-path $P$ in $\gamma(T)$,
which exists as $a\in\gamma(T)$, by Lemma~\ref{ADlem}.
	We claim that 
	\begin{equation}\label{hamlet}
	\begin{minipage}[c]{0.8\textwidth}\em
if $v\in V(T)$ is such that $T_v$ contains a $v$-linked $B_{13}$-tree, 
if $\epsilon(T_v)\subseteq G'$ 
and if 
$x \in V(P)$ and $x$ separates
 $\gamma(T)\cap \epsilon(T_v)$ from $P$ in $\gamma(T)$,
then $x=\gamma(s^*)$. 
	\end{minipage}\ignorespacesafterend 
	\end{equation} 
Suppose that $x\neq \gamma(s^*)$.
As $T_v$ contains a $v$-linked $B_{13}$-tree it follows from 
Lemma~\ref{trees:weightOfBinaryTree}
that $\omega(T_v)\geq 2^3-2\geq 6$. 
From Lemma~\ref{trees:atMostTwo} 
we obtain $\omega_\gamma(\epsilon(T_v))\geq 6-2=4$.

As $x$ separates $P$ from  $\gamma(T)\cap \epsilon(T_v)$
it follows from $x\neq \gamma(s^*)$ that $\gamma(T_{s^*})$ and $\gamma(T)\cap\epsilon(T_v)$ are disjoint. Moreover, $\gamma(T_{s^*})$ and $\epsilon(T_v)$ are both contained in $G'$.
Thus 
\begin{align*}
\omega_\gamma(G') & = |\gamma(L_{10})\cap V(G')| \\
 & \geq |\gamma(L_{10})\cap V(\gamma(T_{s^*}))| + 
|\gamma(L_{10})\cap V(\gamma(T)\cap \epsilon(T_v))| \\
&= \omega(T_{s^*}) + \omega_\gamma(\epsilon(T_v)) \geq  \wmin+4,
\end{align*}
where we have used~\ref{wmin}.
However,~\ref{treeprops:weight} gives  
	$|\epsilon(L_{10})\cap V(G')|=\wmin$,
which means that Lemma~\ref{trees:atMostTwo} yields $\omega_\gamma(G')\leq \wmin+2$.
This contradiction proves~\eqref{hamlet}.

\medskip
	
	From Lemma~\ref{trees:rootMapping} and~\ref{treeprops:u16} we deduce that 
	$\gamma(s^*)\in V(W) $ or $\gamma(s^*)=\epsilon(u_{\BORDER{}}))$.
	In each of the two cases,
we will prove now that $\gamma(T)$ contains an $(a$--$b,c$--$d)$-linkage in $W$.
We point out that both cases are possible (for some choices of~$T$).

	Let us first prove that
\begin{equation}\label{trees:rootAtU\BORDER{}}
\emtext{
	if $\gamma(s^*)=\epsilon(u_{\BORDER{}})$, then $W$ contains an $(a$--$b,c$--$d)$-linkage.
}
\end{equation} 
	In this case, $P$ connects $a$ and $\epsilon(u_{\BORDER})\in\epsilon(C)$
(note that $u_{15}\in V(C)$ by~\ref{ACD})
	and therefore contains an $W$--$\epsilon(C)$-path contained in $Z_b$ or in $Z_c$. 
Suppose it is the latter.
	Then, the path $P$ contains in particular $\epsilon(u_{\SMALLPLUS{}})$.
Now, $u_{13}$ lies in $P_M$ by~\ref{ukmain} and thus has precisely one 
neighbour $v$ that does not lie in $P_M$. By~\ref{linorder},
$T_v$ is disjoint from $P_M$, which means that $\epsilon(T_v)$ is defined,
and moreover that $\epsilon(T_v)$ lies in $G'$. 
Moreover, 
as $u_{13}\in L_{13}$ by~\ref{ukleaves}, we observe that $T_v$ contains a $v$-linked $B_{13}$-tree.
From~\ref{epssep} it follows that $\epsilon(u_{13})$ separates $\gamma(T)\cap\epsilon(T_v)$
from $P$ --- then, however, we obtain a 
contradiction to~\eqref{hamlet}.
	
	Thus,  $P$ contains an $W$--$\epsilon(C)$-path contained in $Z_b$.
	Therefore, $P$ contains an  $a$--$b$-path, which is then contained in $W$.
By Lemma~\ref{ADlem}, $\gamma(T)$ contains an $u_{15}$--$d$-path $Q$. 
If $Q$ and $aPb$ meet then there is a $t$ such that $\gamma(t)\in V(aPb)$ 
and such that $\gamma(t)$ separates $d$, and then also $\epsilon(D)\cap\gamma(T)$,
from $P$ in $\gamma(T)$. As $s^*\notin V(aPb)$, this contradicts~\eqref{hamlet};
where we have used $D=T_{u_{12}}$ by~\ref{ACD} and $u_{12}\in L_{12}$ by~\ref{ukleaves}.
Therefore $aPb$ and $Q$ do not meet, and as $Q$ starts in $\epsilon(u_{15})$
it follows that $Q$ must contain $c$. Then $cQd\subseteq W$, and we have found
the linkage.

	\medskip
	
Finally, we claim that:
	\begin{equation}\label{trees:rootInH}
\emtext{
	if $\gamma(s^*)\in V(W)$, then $W$ contains an $(a$--$b,c$--$d)$-linkage.
}	
\end{equation} 
Let $T_1,T_2$ be the two components of $T_{u_{15}}-u_{15}$ such that 
$T_2$ contains $P_M$ and then also $D$. 
Then, by~\ref{mpath}, it follows that 
$\omega(T_1)\geq\omega(T_2)$, which implies that 
$\sigma(T_{u_{15}})=(\omega(T_1),\omega(T_2))=:(\alpha_1,\alpha_2)$.

As $s^*\in \Smin\subseteq L_{15}$, the two components $R_1,R_2$ of $T_{s^*}-s^*$
both contain $B_{15}$-trees. Since $\epsilon(s^*)\in V(W)$, it follows from 
Lemma~\ref{HDlem} (ii) that each of $\gamma(R_1)$ and $\gamma(R_2)$ 
contains a vertex of $\epsilon(C)$. 
In particular, we may assume that $\gamma(R_1)$ contains $b$
and $\epsilon(u_{15})$, and that $\gamma(R_2)$ contains $c$.

By~\ref{epssep}, $\epsilon(u_{15})\in V(R_1)$ separates $\epsilon(T_1)$ from 
every vertex of degree at least~$3$ in $G-\epsilon(T_1)$. 
Thus, $\gamma(L_{10})\cap \epsilon(T_1)\subseteq V(R_1)$. 

Suppose that $\gamma(R_2)$ is disjoint from $\epsilon(D)$. 
Then $\gamma(R_2)\subseteq W\cup (\epsilon(T_2)\setminus\epsilon(D))$, and consequently,
by Lemmas~\ref{trees:atMostTwo} and~\ref{ADlem},
\begin{align*}
	\omega(R_2)&=|\gamma(L_{10})\cap \gamma(R_2)|
	\leq 2+|\epsilon(L_{10})\cap \gamma(R_2)|\\
	&\leq 2 + |\epsilon(L_{10}) \cap (W\cup (\epsilon(T_2)\setminus\epsilon(D)))|
	\leq 2+ \alpha_2-6<\alpha_2.
\end{align*}
On the other hand, however, 
the choice of $s^*$, see Lemma~\ref{trees:treeOutsideA}, requires that 
$\sigma(T_{s^*})\geq\sigma(T_{u_{15}})=(\alpha_1,\alpha_2)$. 
As $\sigma(T_{s^*})=(\omega(R_1),\omega(R_2))$ or 
$\sigma(T_{s^*})=(\omega(R_2),\omega(R_1))$, we obtain a contradiction.

Therefore, $\gamma(R_2)$ must meet $\epsilon(D)$ and thus contain $d$.
The $c$--$d$-path $Q'$ contained in $\gamma(R_2)$ lies in $W$. 
Moreover, as $\gamma(R_2)$ is disjoint from $P$ and from $\gamma(R_1)$
it does not meet the $a$--$b$-path $P'$ in $\gamma(T)$, as the branch vertices
contained in $P'$ lie in $P\cup \gamma(R_1)$. As $P'$ thus avoids $c$
it follows that $P'\subseteq W$. The pair $P',Q'$ is thus 
the desired linkage.

\medskip
	
	From  \eqref{trees:rootAtU\BORDER{}} and \eqref{trees:rootInH} we directly derive
	Lemma \ref{trees:mainLem}.

\section{Open problems}

We have proved that the subdivisions of all subcubic trees of sufficiently 
large pathwidth do not have the edge-\EP. 
We believe we can also prove 
that the expansions of a sufficiently large grid do not have the edge-\EP. 
Obviously, large grids have large treewidth (and large pathwidth).
Motivated by these results, we conjecture:
	\begin{conjecture}
		There is an integer $c$
		such that for every planar graph $H$ of treewidth (or even pathwidth) at least $c$,
		the family  of $H$-expansions does not have the edge-\EP.
	\end{conjecture}
It is well known that every graph of large treewidth contains an expansion of a large
grid. 
Unfortunately, our argument for grid-expansions to which we alluded above
does not carry over to graphs merely containing a (large) grid-expansion.

\medskip

We also pose a positive conjecture, one about graph classes 
that we believe to have the edge-\EP.
It is striking that for all classes of $H$-expansions that we know have
the edge-property, we can find $H$ as a minor in a sufficiently large 
condensed wall. This is the case for long cycles, for $\theta$-graphs,
as well as for $K_4$. 
We therefore conjecture that containment in the condensed wall
is a sufficient condition:
	\begin{conjecture}
		Let $H$ be a planar graph 
		such that there is an integer $r$
		such that the condensed wall of size $r$
		contains an $H$-expansion.
		Then, the family of $H$-expansions has the edge-\EP.
	\end{conjecture}
	
If we were so lucky that both conjectures turn out to be true, then 
we still would not have a characterisation for which $H$ the family 
of $H$-expansion has the edge-property. That is, we still would not 
have an edge-analogue of Robertson and Seymour's theorem.

Could we perhaps strengthen the first conjecture 
by believing the reverse direction of the second conjecture? 
Namely, that $H$-expansion do not have the edge-\EP\ whenever 
arbitrarily large condensed walls do not contain any $H$-expansion? 
We doubt this is true. If $H$ does not fit in the condensed wall 
but \emph{almost} fits,  then it seems exceedingly difficult 
to pursue a construction as we have done in Sections~\ref{sec:ladder} 
and~\ref{sec:trees}.

\bibliographystyle{amsplain}
\bibliography{../erdosposa}

\vfill

\small
\vskip2mm plus 1fill
\noindent
Version \today{}
\bigbreak

\noindent
Henning Bruhn
{\tt <henning.bruhn@uni-ulm.de>}\\
Matthias Heinlein
{\tt <matthias.heinlein@uni-ulm.de>}\\
Institut f\"ur Optimierung und Operations Research\\
Universit\"at Ulm\\
Germany\\

\noindent
Felix Joos
{\tt <f.joos@bham.ac.uk>}\\
School of Mathematics\\ 
University of Birmingham\\
United Kingdom

\end{document}